\documentclass[a4paper,11pt]{article}

\usepackage[cp1250]{inputenc}
\usepackage{amssymb}
\usepackage[T1]{fontenc}
\usepackage{amsmath}
\usepackage{amsthm}
\usepackage{accents}
\usepackage{enumerate}
\usepackage{bbm}
\usepackage[left=3cm,right=3cm,top=3cm,bottom=3cm]{geometry}
\usepackage{hyperref}
\hypersetup{colorlinks=true, linkcolor=blue, citecolor=red,
filecolor=black, urlcolor=black }
\newcommand{\R}{\mathbb{R}}
\newcommand{\N}{\mathbb{N}}
\newcommand{\F}{\mathcal{F}}
\newcommand{\Fb}{\mathbb{F}}
\newcommand{\T}{\mathcal{T}}
\newcommand{\Tp}{{^p\mathcal{T}}}

\newcommand{\V}{\mathcal{V}}
\newcommand{\Vp}{{^p\mathcal{V}}}
\newcommand{\Vo}{\mathcal{V}_0}
\newcommand{\Vc}{\mathcal{V}^1}
\newcommand{\M}{\mathcal{M}}

\newcommand{\A}{\mathcal{A}}

\newcommand{\ind}[1]{\mathbbm{1}_{#1}}
\newcommand{\rs}{RBSDE}
\newcommand{\lrs}{\underline{RBSDE}}
\newcommand{\urs}{\overline{RBSDE}}
\newcommand{\ors}{ORBSDE}

\DeclareMathAccent{\wtilde}{\mathord}{largesymbols}{"65}

\DeclareMathOperator*{\esssup}{ess\,sup}

\newcommand{\cadlag}{c\`adl\`ag }

\theoremstyle{definition}
\newtheorem{ex}{Example}[section]
\newtheorem*{defin}{Definition}
\newtheorem{rem}[ex]{Remark}
\theoremstyle{plain}
\newtheorem{tw}[ex]{Theorem}
\newtheorem{lm}[ex]{Lemma}
\newtheorem{wn}[ex]{Corollary}
\newtheorem{stw}[ex]{Proposition}

\makeatletter
\renewcommand\maketitle{\par
  \begingroup
    \renewcommand\thefootnote{\@fnsymbol\c@footnote}%
    \def\@makefnmark{\rlap{\@textsuperscript{\normalfont}}}%
    \long\def\@makefntext##1{\parindent 1em\noindent
            \hb@xt@1.8em{%
                \hss\@textsuperscript{\normalfont}}##1}%
    \if@twocolumn
      \ifnum \col@number=\@ne
        \@maketitle
      \else
        \twocolumn[\@maketitle]%
      \fi
    \else
      \newpage
      \global\@topnum\z@   % Prevents figures from going at top of page.
      \@maketitle
    \fi
    \thispagestyle{plain}\@thanks
  \endgroup
  \setcounter{footnote}{0}%
  \global\let\thanks\relax
  \global\let\maketitle\relax
  \global\let\@maketitle\relax
  \global\let\@thanks\@empty
  \global\let\@author\@empty
  \global\let\@date\@empty
  \global\let\@title\@empty
  \global\let\title\relax
  \global\let\author\relax
  \global\let\date\relax
  \global\let\and\relax
}
\makeatother

\numberwithin{equation}{section}

\title{Systems of BSDEs with oblique reflection and  related optimal switching problems}

\author{Mateusz Topolewski \footnote{M. Topolewski: Faculty of Mathematics and Computer Science, Nicolaus Copernicus University, Chopina 12/18, \mbox{87--100 Toruń}, Poland. {E-mail: \tt woland@mat.umk.pl}}}
\begin{document}
\date{\scriptsize}
\maketitle
\begin{abstract}
We consider systems of backward stochastic differential  equations
with \cadlag upper barrier $U$ and oblique reflection from below driven by
an increasing continuous function $H$. Our equations are defined
on general probability spaces with a filtration satisfying merely
the usual assumptions of right continuity and completeness. We
assume that the pair $(H(U),U)$
satisfies a Mokobodzki--type condition. We prove the existence of a
solution for integrable terminal conditions and integrable
quasi--monotone generators. Applications to the optimal switching
problem are given.
\end{abstract}


\vspace{5pt}


\section{Introduction and notation}
\label{sec1}

In this paper, we study the problem of existence and uniqueness of a solution of system of backward stochastic differential equations (BSDEs for short) with oblique reflection from below and fixed upper barrier. The main new feature is that we deal with equations on probability spaces with general filtration $\Fb =\{\F_t, t \in [0, T]\}$ satisfying only the usual conditions of right-continuity and completeness. Moreover, we deal with equations with  $L^1$-data.

Let $T>0$ and $d\in\N.$ Suppose we are given an $\F_T$-measurable random vector $\xi=(\xi^1,\xi^2,\ldots,\xi^d)$, a progressively measurable function $f=(f^1,f^2,\ldots,f^d):\Omega\times[0,T]\times\R^d\to\R^d$, an adapted $\R^d$--valued \cadlag
process $U=(U^1,U^2,\ldots,U^d)$ and a function $H=(H^1,H^2,\ldots,H^d):\Omega\times[0,T]\times\R^d\to\R^d$ such that $H^j(y)$ does not depend on $y^j$ for every $y\in\R^d$. Roughly speaking, by a solution of a system of BSDEs with terminal condition $\xi$, generator $f$, oblique reflection driven by $H$ and upper barrier $U$, we mean a~quadruple $(Y,M,K,A)=\{(Y^j,M^j,K^j,A^j)\}_{j=1,\ldots,d}$ of \cadlag adapted processes such that $Y$
is of Doob's class D, $M$ is a local martingale with $M_0=0$, $K,A$ are increasing processes such that
$K_0=A_0=0$, and a.s. we have
\begin{equation}\label{eq.1.1}
\begin{cases}
Y^j_t=\xi^j+ \int_t^T f^j(r,Y_r)\,dr +\int_t^T dK^j_r -
\int_t^T dA^j_r - \int_t^T dM^j_r,& t\in[0,T],\smallskip\\
\int_0^T(Y^j_{r-}-H^j_{r-}(Y_{r-}))\,dK^j_r=\int_0^T (U^j_{r-}-Y^j_{r-})\,dA^j_r=0,\smallskip\\
H^j_t(Y_t)\leq Y^j_t\leq U^j_t, \qquad t\in[0,T]
\end{cases}
\end{equation}
for $j=1,\ldots,d$ (see Section \ref{sec3} for detailes). In the case where $d=1$,  \eqref{eq.1.1}  reduces to the one-dimensional reflected BSDE with upper barrier $U$, which was
thoroughly  investigated in Klimsiak \cite{klimsiak2014}.
Therefore, in the present paper, we consider the case where $d\ge2$.

Our motivation for considering such a general setting comes from the
theory of stochastic control (the optimal switching problem). Let us consider a power station that  can
produce electricity in one of $d$ modes which can be switched in
time. Let $\{\theta_n\}$ be an increasing sequence of stopping
times (switching times) such that $N=\inf\left\{ n\in\N\;:\;
\theta_n=T \right\} <\infty$ $P$-a.s. Define an admissible
switching control (switching strategy) as a stochastic process of
the form
\begin{equation}\label{eq.1.3}
a(t) = \sum_{n=0}^{N-1}\alpha_n\ind{[\theta_n,\theta_{n+1})}(t) + \alpha_N\ind{\{\theta_N\}}(t),\quad t\in[\theta_0,T],
\end{equation}
where $\alpha_n$ is an $\F_{\theta_n}$-measurable random variable
with values in $\{1,\ldots,d\}$. By $\A^j_t$ we denote the set
of all admissible switching controls with the initial data
$(\alpha_0,\theta_0)=(j,t)$. Assume  that the price of electricity
is given by some process $X$ and in the $j$-th mode the income of
the station is driven by a function $f^j(\cdot,X)$. In the classical model, the profitability
$R^{(a)}_0$  of the power station on the interval $[0,T]$ under a~switching strategy $a$ is given by
\begin{equation*}
R_0^{(a)}:= R_0^{(a)}(X)= \int_0^T f^{a(r)}(r,X_r)\,dr - \sum_{n=1}^{N} c^{\alpha_{n-1}, \alpha_n}(\theta_n),
\end{equation*}
where $c^{j,k}$,  $j,k=1,\ldots,d$, are some positive $\Fb^X$-adapted processes. The values $c^{\alpha_{n-1}, \alpha_n}(\theta_n)$ can be regarded as costs of switching from  mode $\alpha_{n-1}$ to $\alpha_n$ at time $\theta_n$.

In this paper, we consider a modification of the above  classical
model. In our model the profitability is limited by some external
upper barrier. More precisely, we assume that for a switching
strategy $a\in\A^j_t$ of the form \eqref{eq.1.3} the profit of the
power station is given by the first component of the solution
$(R^{(a)}, M^{(a)},D^{(a)})$ of the reflected BSDE of the form
\begin{equation}\label{eq.1.4}
\begin{cases}
R_s^{(a)}= \xi^{a(T)}+ \int_s^T f^{a(r)}(r,X_r)\,dr\\
\qquad\quad -\sum_{n=1}^{N} c^{\alpha_{n-1}, \alpha_n}(\theta_n)\ind{(s,T]}(\theta_n)  - \int_s^T dD^{(a)}_r - \int_s^T dM^{(a)}_r,\quad s\in[t,T],\smallskip\\
R^{(a)}_s\leq U^{a(s)}_s,\quad s\in[t,T],\smallskip\\
\int_t^T (U^{a(r-)}_{r-} - R^{(a)}_{r-})\, dD^{(a)}_r=0.
\end{cases}
\end{equation}
In (\ref{eq.1.4}),  $U=(U^1,U^2,\ldots,U^d)$ is a given $\Fb^X$-adapted process such that $U^{i,-}=(-U^i)\vee0$, $i=1,\dots,d$, is of class D. If
\begin{equation}\label{eq.1.5}
E\Big(\int_0^T |f^{a(r)}(r,X_r)|\,dr + \sum_{n=1}^{N} |c^{\alpha_{n-1}, \alpha_n}(\theta_n)|\Big)<\infty,
\end{equation}
then the existence and uniqueness of a solution of \eqref{eq.1.4}
is ensured by \cite[Theorem 2.3(ii)]{MT2016}.
Since the process $U$ is $\Fb^X$-adapted, without loss of generality one can assume that $U=U(X)$ ($U$ is functionally dependent on $X$). The barrier $(U_t^{a(t)})_{t\in[0,T]}$ can be interpreted as an upper limit for the profitability when one uses the control $a$. This limitation can be a~result of some external economic factors (production-possibilities and storage capacity may limit the profits) or legal regulations, see, e.g., Brenann and Schwartz \cite{Brenann1986}, Dixit and Pindyck \cite[Chapter 9]{Dixit1994}, Grenadier \cite{Grenadier1996,Grenadier1999,Grenadier2002}, McDonald and Siegel \cite{McDonald1986}.
The main goal is
to find, for given $t\in[0,T]$, a~strategy $\widehat a$ which
maximizes the profitability of the power station, i.e. a strategy
$\widehat a\in\A^j_t$ such that
\[
R^{(\widehat a)}_t=\esssup_{a\in\A^j_t} R^{(a)}_t,
\]
for $t\in[0,T]$ and $j=1,2,\ldots,d.$

In the paper we  show that the above optimal stopping  problem  is
closely related to systems of equations of the form (\ref{eq.1.1}) with
filtration generated by the price process $X$ and $H$ defined as
\begin{equation}\label{eq.1.2}
H^j_t(y)=\max_{k\neq j} (y^k-c^{j,k}(t)),\quad j=1,\ldots,d.
\end{equation}
In (\ref{eq.1.2}), $\{c^{j,k}\}_{j,k=1,\ldots, d}$ are switching costs considered in \eqref{eq.1.4}, and they are assumed to be continuous.
In general, the price process $X$ may have jumps (for instance, $X$ can be modelled by a L\'evy process). This is why it is important to consider equations of the form (\ref{eq.1.1}) with general filtration.

To our knowledge,
the paper by Tang, Zhong and Koo \cite{TangZhongKoo2013} is the only paper devoted to  equations with oblique reflection from below and fixed upper barrier. In \cite{TangZhongKoo2013}  the existence and uniqueness of a solution to \eqref{eq.1.1} is showed under the assumption that the data are $L^2$-integrable, $f$ is Lipschitz continuous with respect to the space variable $y$ and the upper barrier $U$ is continuous. In \cite{TangZhongKoo2013} it is also assumed that the underlying filtration is Brownian. Problem (\ref{eq.1.1})  without upper barrier is considered in  \cite{HamJean2007,HamZhang2010,HuTang2010,Klimsiak2016}. In these papers the setting is similar to that in \cite{TangZhongKoo2013}. In particular, only a Brownian filtration is considered.
An interesting,  different approach to optimal switching problem with constants switching costs and related reflected BSDEs  is presented in the recent paper by Chassagneux and Richou \cite{CH}.
As a matter of fact, in \cite{CH} much more general than in \cite{HamJean2007,HamZhang2010,HuTang2010} equations with Brownian filtration and oblique reflection are considered. Klimsiak \cite{Klimsiak2016} studied (\ref{eq.1.1}) without upper barrier, with  general filtration $\Fb$, $L^1$-data and quasi--monotone generator $f$. Moreover,  he considers  oblique functions more general than \eqref{eq.1.2}. In \cite{Klimsiak2016} it is proved that for the existence result it suffices to assume that $H^j$ is continuous and increasing with respect to $y$ for $j=1,\ldots,d.$ If moreover $H$ is of the form (\ref{eq.1.2}), then the solution is unique.

Our main theorem states that if $f,\xi,H$ satisfy the assumptions adopted in \cite{Klimsiak2016} and $U$ is a~\cadlag process such that $H^j(U),U^j$ satisfy the so-called Mokobodzki condition for $j=1,\ldots,d$ (see hypothesis (M)), then there exists a solution of \eqref{eq.1.1}. Thus, we generalize the existence result of \cite{Klimsiak2016} to equations with upper barrier, and at the same time we generalize the results of  \cite{TangZhongKoo2013} in the sense that we consider problem (\ref{eq.1.1}) with general filtration and much weaker assumptions on the data.
Like in \cite{TangZhongKoo2013},
the existence of a solution to (\ref{eq.1.1}) is proved by the Picard iteration method. However, because of the general filtration and weak assumption on the data, our proof is more involved. Also note that in our proof we use in an essential way  some results on  one-dimensional reflected BSDEs with general filtration proved in \cite{klimsiak2014,MT2016}.
We are not able to prove that the solution to (\ref{eq.1.1}) is unique for general $H$. However, we show that the uniqueness for (\ref{eq.1.1}) holds true if $f^j(t,y)$ does not depend on $y^1,\ldots,y^{j-1},y^{j+1},\ldots,y^d$ and $H$ is of  the form \eqref{eq.1.2}. This is done by proving Proposition \ref{prop.4.4}, which links  solutions of \eqref{eq.1.1} with  solutions of some optimal switching problem.

This paper is organized as follows. Section \ref{sec2} contains a short review of properties of one--dimensional reflected BSDEs with one and two barriers. In Section \ref{sec3} we consider systems of BSDEs with oblique reflection and prove of the main existence result. Finally, in Section \ref{sec4} we give an application of the results of Section \ref{sec3} to the optimal switching problem and state the uniqueness result.

{\bf Notation.} Let  $T>0$, and let
$(\Omega,\F,\Fb=\{\F_t\}_{t\in[0,T]},P)$  be a filtered
probability space with filtration satisfying  the usual
assumptions of completeness and right continuity. By $\T$  we
denote the set of all $\Fb$--stopping times such that $\tau\leq
T$, and by $\T_t$, $t\in[0,T]$, the set of $\tau\in\T$ such that
$P(\tau\geq t)=1$.
By $\V$ we denote the set of all $\Fb$-progressively measurable
processes of finite variation, and by $\Vc$ the subset of $\V$
consisting of all processes $V$ such that $E|V|_T<\infty$, where
$|V|_T$ stands for the variation of $V$ on $[0,T]$. $\Vo$ is  the
subset of $\V$ consisting of all  processes $V$ such that $V_0=0$,
$\Vo^+$ (resp. $\Vp_0^+$) is the subset of $\Vo$ of all increasing
processes (resp. predictable increasing processes). We also define $\Vc_0=\Vc\cap\V_0.$ $\M$ (resp.
$\M_{loc}$) denotes the set of all $\Fb$-martingales (resp. local
martingales). By $L^1(\Fb)$ we denote the space of all
$\Fb$-progressively measurable processes $X$ such that $E\int_0^T
|X_t|dt<\infty$, and by $L^{1}(\F_T)$ the  space of all
$\F_T$-measurable random variables $\xi$ such that
$E|\xi|<\infty$.

\section{One-dimensional reflected BSDEs}
\label{sec2}

For  completeness of exposition and future reference, in this section we gather some results on one-dimensional reflected BSDEs with one and two reflecting barriers. These results are taken mainly from \cite{klimsiak2014,Klimsiak2016} (with some minor modifications).

In what follows $L$ is some \cadlag process, $\xi$ is $\F_T$--measurable random variable such that $\xi\geq L_T$, $V\in\V_0$ and $f:\Omega\times[0,T]\times\R\to\R$ is a measurable function such that for every $y\in\R$ the process $f(\cdot,y)$ is $\Fb$-progressively measurable. The following assumptions will be needed throughout this section. Similar assumptions were considered in \cite{klimsiak2014, MT2016}.
\begin{enumerate}[(H1)]
\item For almost
every $t\in[0,T]$ and all $y,y'\in\R$,
\[
(f(t,y)-f(t,y'))(y-y')\leq 0,
\]
\item $\int_0^T|f(r,y)|\,dr<\infty$ for every $y\in\R$,
\item the function $\R\ni y\mapsto f(t,y)$ is continuous  for almost
every $t\in[0,T]$,
\item $\xi\in L^1(\F_T)$, $V\in\Vo\cap\V^1$, and there exists a~\cadlag semimartingale $S$ being a~difference of supermartingales of class {D} such that
\begin{equation}\label{eq.3.1}
E\int_0^T |f(r,S_r)|\,dr<\infty.
\end{equation}
\end{enumerate}

\begin{rem}\label{rem.3.1}
Let $S$ be a \cadlag semimartingale which is a difference of supermartingales of class D. Then there exist $C\in\Vc_0$ and $N\in\M$ such that $N_0=0$ and
\begin{equation}\label{eq.3.01}
S_t=S_0+C_t+N_t.
\end{equation}
Indeed, let $S=S^1 - S^2$, where $S^1,S^2$ are some supermartingales of class D. Then, by the Doob--Meyer decomposition (see \cite[Theorem 2]{Meyer1962} or \cite[Theorem III.11]{Protter}), there exist $C^1,C^2\in\V_0^+\cap\Vc$ and $N^1,N^2\in\M$ with $N^1_0=N^2_0=0$ such that
\[
S^1_0=S^1_0-C^1_t+N^1_t,\qquad S^2_0=S^2_0-C^2_t+N^2_t, \qquad t\in[0,T].
\]
Setting $C=-C^1+C^2$ and $N=N^1-N^2$ yields (\ref{eq.3.01}).
\end{rem}

\begin{defin}
We say that a triple $(Y,M,K)$ of \cadlag processes is a
solution  of the reflected BSDE with terminal condition $\xi$,
generator $f+dV$ and lower barrier $L$ (we write $\lrs(\xi,f+dV,L)$ for short) if
\begin{enumerate}[(i)]
\item $Y$ is a process of class D, $K\in\Vp_0^+$,
$M\in\M_{loc}$ with $M_0=0$,
\item $L_t\leq Y_t$ for $t\in[0,T]$, $\Pr$-a.s.,
\item $\int_0^T(Y_{r-}-L_{r-})\,dK_r=0$,
\item
$Y_t=\xi+ \int_t^T f(r,Y_r)\,dr + \int_t^T dV_r +\int_t^T dK_r- \int_t^T dM_r$, $t\in[0,T]$, $P$-a.s.
\end{enumerate}
\end{defin}

The following theorem shows how to prove assertions of \cite[Theorem 2.13]{klimsiak2014} under modified assumptions.

\begin{stw}\label{stw.2.1}
Assume that $\xi,f,V$ satisfy \emph{(H1)--(H4)} and there exists a process $X$ being a difference of supermartingales of class {\em D} such that $X_t\geq L_t$ for $t\in[0,T]$ and $E\int_0^T |f(r,X_r)|\,dr<\infty$. Then there exists a solution $(Y,M,K)$ of $\lrs(\xi,f+dV,L)$ such that $\int_0^T |f(r,Y_r)|\,dr<\infty$.
\end{stw}
\begin{proof}
We begin with showing that $f_S=f(\cdot, S+\cdot)$  satisfies
hypotheses (H1)--(H3) and (H4) with the semimartingale $S=0$. By (H1),
\[
(f(t,S_t+y)-f(t,S_t+y'))(y-y')\leq 0, \qquad t\in[0,T], \; y,y'\in\R.
\]
Set $\underline{S}=\inf_{t\in[0,T]}S_t$ and
$\overline{S}=\sup_{t\in[0,T]}S_t$. Since $S$ has \cadlag paths,
$\underline{S},\overline{S}$ are finite. Since $f$ satisfies (H1),
\[
f(t,\overline{S}+y)\leq f(t,S_t+y)\leq f(t,\underline{S}+y),\qquad t\in[0,T],\; y\in\R.
\]
By this and (H2), $f_S$ satisfies (H2). Clearly, $f_S$ also satisfies  (H3) and (H4).
Note that $X-S$ is a difference of supermartingales of class D such that $E\int_0^T |f_S(r,X_r-S_r)|\,dr<\infty$ and, by Remark \ref{rem.3.1}, $S$ admits  decomposition \eqref{eq.3.01} with $C\in\V^1_0$ and $N\in\M$. Therefore, by \cite[Theorem 2.13]{klimsiak2014}, there exists a solution $(Y^S,M^S,K^S)$ of the equation $\lrs(\xi-S_T,f_S+dV+dC,L-S)$, i.e.
\[
Y^S_t=\xi-S_T + \int_t^T f(r,S_r+Y^S_r)\,dr + \int_t^T d(V_r+C_r) +\int_t^TdK^S_r -\int_t^TdM^S_r.
\]
Moreover, $E\int_0^T |f_S(r,Y^S_r)|\,dr<\infty$. Write $(Y,M,K)=(Y^S+S,M^S+N,K^S)$. Then
\[
Y_t=\xi+ \int_t^T f(r,Y_r)\,dr + \int_t^T dV_r +\int_t^T dK_r- \int_t^T dM_r.
\]
We also have
\[
\int_0^T (Y_{r-}-L_{r-})\,dK_r=\int_0^T (Y^S_{r-} -(L_{r-}-S_{r-}))\,dK^S_r=0
\]
and $E\int_0^T |f(r,Y_r)|\,dr=E\int_0^T |f_S(r,Y^S_r)|\,dr<\infty,$
which shows that $(Y,M,K)$ is a solution of $\lrs(\xi,f+dV,L)$ and proves the proposition.
\end{proof}
\begin{rem}\label{rem.2.1}
(i) Let $(Y,M,K)$ be a solution of an equation $\lrs(\xi,f+dV,L)$.  For $\tau\in\T$ set $f^\tau(t,y)=f(t,y)\ind{[0,\tau]}$. Then the triple $(Y_{\cdot\wedge\tau},M_{\cdot\wedge\tau},K_{\cdot\wedge\tau})$ is a solution of $\lrs(Y_\tau, f^\tau+dV_{\cdot\wedge\tau},L_{\cdot\wedge\tau})$. To show this it is  enough to repeat step by step the proof of \cite[Lemma 3.5]{MT2016}. \smallskip\\
(ii) Assume additionally that
\begin{equation}\label{eq.2.1}
E\Big( \int_0^\tau |f(r,Y_r)|\,dr + |V|_\tau \Big)<\infty.
\end{equation}
Then for any $t\in[0,T]$,
\begin{equation}\label{eq.2.2}
Y_{t\wedge\tau}=\esssup_{\sigma\in\T_t}E\Big(\int_{t\wedge\tau}^{\sigma\wedge\tau}f(r,Y_r)\,dr + \int_{t\wedge\tau}^{\sigma\wedge\tau} dV_r + L_{\sigma}\ind{\{\sigma<\tau\}}+Y_\tau\ind{\{\sigma\geq\tau\}}\Big|\F_t\Big).
\end{equation}
Indeed, the triple $(Y_{\cdot\wedge\tau}, M_{\cdot\wedge\tau},K_{\cdot\wedge\tau})$ is a solution of $\lrs(Y_\tau, f^\tau+dV_{\cdot\wedge\tau},L_{\cdot\wedge\tau})$ on $[0,\tau]$ in the sense of \cite[Definition 3.5]{Klimsiak2016}. %with $T=\tau$.
Therefore (\ref{eq.2.2}) follows from  \cite[Remark 3.7]{Klimsiak2016}.
\end{rem}

\begin{rem}\label{rem.2.2}
Let $f,V,L$ satisfy the assumptions of Proposition \ref{stw.2.1}. Assume that $Y$ is a~process of class D such that \eqref{eq.2.1} and \eqref{eq.2.2} are satisfied. Then there exist processes $K,M$ such that $(Y_{\cdot\wedge\tau},M_{\cdot\wedge\tau},K_{\cdot\wedge\tau})$ is a solution of $\lrs(Y_\tau, f^\tau+dV_{\cdot\wedge\tau},L_{\cdot\wedge\tau})$. In particular,
\[
\int_0^\tau(Y_{r-}-L_{r-})dK_r=0.
\]
Indeed, by \eqref{eq.2.2}, \cite[App. I.22 Theorem]{DellMey2} and the Doob--Meyer decomposition theorem, there exist $K\in\Vp_0^{+}$ and $M\in\M$ such that
\[
Y_{t\wedge\tau}+\int_0^{t\wedge\tau}f(r,Y_r)\,dr + \int_0^{t\wedge\tau}dV_r=Y_0 - K_{t\wedge\tau} + M_{t\wedge\tau},\quad t\in[0,T].
\]
Recall that the assertions of \cite[Theorem 2.13]{klimsiak2014} remains true under the assumptions of Proposition \ref{stw.2.1} (see the proof of Proposition \ref{stw.2.1} for details). By \cite[Theorem 2.13, Corollary 2.2]{klimsiak2014}, there exists a~unique solution $(Y',M',K')$ of $\lrs(Y_\tau, f^\tau+dV_{\cdot\wedge\tau},L_{\cdot\wedge\tau})$. By Remark \ref{rem.2.1}(i), $(Y',M',K')=(Y'_{\cdot\wedge\tau},M'_{\cdot\wedge\tau},K'_{\cdot\wedge\tau})$. Therefore, by \cite[Lemma 2.8, Theorem 2.13]{klimsiak2014}, $Y'_{t\wedge\tau}\leq Y_{t\wedge\tau}$ $P$-a.s. for $t\in[0,T]$.
Moreover, by \cite[Theorem 2.13, Corollary 2.9]{klimsiak2014}, \eqref{eq.2.2} holds with $Y$ replaced by $Y'$. By this and (H1),
\[
Y_{t\wedge\tau}-Y'_{t\wedge\tau}\leq \esssup_{\sigma\in\T_t} E\Big(\int_{t\wedge\tau}^{\sigma\wedge\tau}\big(f(r,Y_r)-f(r,Y'_r)\big)\,dr \Big|\F_t\Big)\leq 0,\quad t\in[0,T],
\]
so in fact $Y_{t\wedge\tau}=Y'_{t\wedge\tau}$, $t\in[0,T]$. Hence, by uniqueness of the Doob-Meyer decomposition, $(Y'_{\cdot\wedge\tau},M'_{\cdot\wedge\tau},K'_{\cdot\wedge\tau})
=(Y_{\cdot\wedge\tau},M_{\cdot\wedge\tau},K_{\cdot\wedge\tau}),$ which completes the proof of the desired assertions.
\end{rem}

Let $U$ be a \cadlag process such that $L_t\leq U_t$ for $t\in[0,T]$ and $L_T\leq\xi\leq U_T.$
\begin{defin}
We say that a quadruple $(Y,M,K,A)$ of \cadlag processes is a
solution  of the reflected BSDE with terminal condition $\xi$,
generator $f+dV$, lower barrier $L$ and upper barrier $U$
($\rs(\xi,f+dV,L,U)$ for short) if
\begin{enumerate}[(i)]
\item $Y$ is a process of class D, $A,K\in\Vp_0^+$,
$M\in\M_{loc}$ with $M_0=0$,
\item $L_t\leq Y_t\leq U_t$ for $t\in[0,T]$, $\Pr$-a.s.,
\item $\int_0^T(Y_{r-}-L_{r-})\,dK_r=\int_0^T
(U_{r-}-Y_{r-})\,dA_r=0$,
\item
$Y_t=\xi+ \int_t^T f(r,Y_r)\,dr + \int_t^T dV_r +\int_t^T d(K_r
-A_r)- \int_t^T dM_r$, $t\in[0,T]$, $P$-a.s.
\end{enumerate}
\end{defin}

The following fact shows how to prove assertions of \cite[Theorem 4.2]{klimsiak2014} in our setting.

\begin{stw}\label{stw.3.1}
Assume that $\xi,f,V$ satisfy \emph{(H1) -- (H4)}.
If there exists a \cadlag semimartingale $X$ being a difference of two supermartingales of class {\em D} and such that $L_t\leq X_t\leq U_t$ for $t\in[0,T]$, then there exists a solution of $\rs(\xi,f+dV,L,U).$
\end{stw}
\begin{proof}
Let $f_S,C,N$ be defined as in the proof of Proposition \ref{stw.2.1}.
By the assumptions, $L_t-S_t\leq X_t-S_t\leq U_t-S_t$ for $t\in[0,T]$ and $X-S$ is a difference of supermartingales od class D. Therefore, by \cite[Theorem 4.2]{klimsiak2014},  there exists a solution $(Y^S,M^S,K^S,A^S)$ of $\rs(\xi-S_T,f_S+dV+dC,L-S,U-S)$, i.e.
\[
Y^S_t=\xi-S_T + \int_t^T f(r,S_r+Y^S_r)\,dr + \int_t^T d(V_r+C_r) +\int_t^Td(K^S-A^S)_r -\int_t^TdM^S_r.
\]
Write $(Y,M,K,A)=(Y^S+S,M^S+N,K^S,A^S)$. Then
\[
Y_t=\xi+ \int_t^T f(r,Y_r)\,dr + \int_t^T dV_r +\int_t^T d(K
-A)_r- \int_t^T dM_r.
\]
Moreover,
\[
\int_0^T (Y_{r-}-L_{r-})\,dK_r=\int_0^T (Y^S_{r-} -(L_{r-}-S_{r-}))\,dK^S_r=0
\]
and
\[
\int_0^T (U_{r-}-Y_{r-})\,dK_r=\int_0^T ((U_{r-}-S_{r-})-Y^S_{r-})\,dA^S_r=0,
\]
which shows that $(Y,M,K,A)$ is a solution of $\rs(\xi,f+dV,L,U)$.
\end{proof}

\begin{defin}
We say that $(Y,M,A)$ is a solution of $\urs(\xi,f+dV,U)$ if $(-Y,-M,A)$ is a solution of $\lrs(-\xi,-f(\cdot,-\cdot)-dV, -U)$.
\end{defin}

\begin{lm}\label{lm.3.3}
Let the assumptions of Proposition \ref{stw.3.1} hold, and let  $(Y^i,M^i,K^i,A^i)$, $i=1,2$, be a~solution of $\rs(\xi^i, f^i+dV^i, L^i,U)$. Assume that $\xi^1\leq\xi^2$, $dV^1\leq dV^2$, $L^1_t\leq L^2_t$ for $t\in[0,T]$. If $f^1(t,y)\leq f^2(t,y)$ for $t\in[0,T]$, $y\in\R$, then $dA^1\leq dA^2.$
\end{lm}
\begin{proof}
Without loss  of generality we may assume that $S\equiv 0$ (see the proof of Proposition \ref{stw.3.1} for details).
Let $(Y^{i,n},M^{i,n},A^{i,n})$ be a solution of $\urs(\xi^i, f^i+n(L^i-\cdot)^+ + dV^i,U)$ for $n\in\N$ and $i=1,2.$ For $k,n\in\N$ we set
\[
\begin{split}
\tau^n_k=\inf\Big\{t>0: \int_0^t \Big( &|f^1(r,Y^{1,n}_r)|+ n(L^1_r-Y^{1,n}_r)^+\\
& + |f^2(r,Y^{2,n}_r)|+ n(L^2_r-Y^{2,n}_r)^+\Big)\,dr > k\Big\} \wedge T.
\end{split}
\]
By Remark \ref{rem.2.1}(i) and \cite[Proposition 2.14]{klimsiak2014},
\[
dA_{\cdot\wedge\tau^n_k}^{1,n}\leq dA_{\cdot\wedge\tau^n_k}^{2,n},\qquad k,n\in\N.
\]
Letting $k\to\infty$ yields
\begin{equation}\label{eq.3.12}
dA^{1,n}\leq dA^{2,n},\qquad n\in\N.
\end{equation}
By \cite[Theorem 4.2]{klimsiak2014}, $A^{i,n}_t\nearrow A^i_t$ for $t\in[0,T]$ and $i=1,2.$ From this and \eqref{eq.3.12} it follows that $dA^1\leq dA^2.$
\end{proof}

\section{Systems of BSDEs with oblique reflection}
\label{sec3}

In what follows $\xi=(\xi^1,\ldots,\xi^d)$ is an $\F_T$-measurable random vector, $V=(V^1,\ldots,V^d)$ is an $\R^d$--valued finite variation process, $f:\Omega\times[0,T]\times\R^d\to\R^d$ is a measurable function such that for every $y\in\R^d$ the process $f(\cdot,y)$ is $\Fb$-progressively measurable. We adopt the following notation
\[
f^j(t,y;c)=f^j(t,y^1,y^2,\ldots,y^{j-1},c,y^{j+1},\ldots,y^d),\qquad y\in\R^d,\quad c\in\R.
\]
For $x,y\in\R^d$ we write $x\leq y$ if $x^j\leq y^j$ for $j=1,\ldots,d.$

The following assumptions will be needed in the rest of the paper. They were considered in  Klimsiak  \cite{Klimsiak2016}.
\begin{itemize}
\item[(A1)] $\xi^j\in L^1(\F_T)$ and $V^j\in\Vc_0$ for $j=1,\ldots,d$.
\item[(A2)] $f(t,\cdot)$ is on--diagonal decreasing for $t\in[0,T]$, i.e. for $j=1,\ldots,d$ we have
\[
f^j(t,y;c)\leq f^j(t,y;c'),\qquad c\geq c', \quad y\in\R^d.
\]
\item[(A3)] $f(t,\cdot)$ is off--diagonal increasing for $t\in[0,T]$, i.e. for $j=1,\ldots,d$ we have
\[
f^j(t,y)\leq f^j(t,y'),\qquad y\leq y',\quad y^j={y'}^{j}.
\]
\item[(A4)] $y\mapsto f(t,y)$ is continuous for every $t\in[0,T]$.
\item[(A5)] $\int_0^T |f^j(r,y)|\,dr<\infty$ for all $y\in\R^d$ and $j=1,\ldots,d$.
\end{itemize}

Let $U$ be an $\R^d$-valued \cadlag processes such that $\xi\leq U_T$.
\begin{defin}
We say that a triple $(Y,M,A)=\{(Y^j,M^j,A^j)\}_{j=1,\ldots, d}$ of adapted $\R^d$--valued \cadlag processes is a solution of the reflected BSDE with upper barrier $U$, generator $f+dV$ and terminal condition $\xi$ ($\urs(\xi,f+dV,U)$ for short) if
\begin{itemize}
\item[(i)] $Y^j$ is of class D, $A^j\in\Vp^+_0$, $M^j\in\M_{loc}$ with $M^j_0=0$ for $j=1,\ldots,d$,
\item[(ii)] $Y_t\leq U_t$ for $t\in[0,T],$
\item[(iii)] $\int_0^T (U^j_{r-}-Y^j_{r-})\,dA^j_r=0$ for $j=1,\ldots,d$,
\item[(iv)] $Y^j_t=\xi^j+\int_t^T f^j(r,Y_r)\,dr+ \int_t^TdV^j_r-\int_t^T dA^j_r - \int_t^T dM^j_r$ for all $t\in[0,T]$ and $j=1,\ldots,d$.
\end{itemize}
\end{defin}

Our motivations for considering reflected equations involving a
finite variation process $V$  comes from the theory of partial
differential equations with measure data. In these applications
$V$ is an additive functional of a Markov process in the Revuz
correspondence with some smooth measure (see
\cite{Klimsiak2016,kliroz,kliroz2015}).

\begin{defin}
We say that a triple $(\underline{Y},\underline{M},\underline{A})=\{(Y^j,M^j,A^j)\}_{j=1,\ldots, d}$ is a subsolution of ${\urs(\xi,f+dV,U)}$ if there exist $\underline{\xi}=(\underline{\xi}^1,\ldots,\underline{\xi}^d)$ and $\underline{V}=(\underline{V}^1,\ldots,\underline{V}^d)$ such that $\underline{\xi}^j\in L^1(\F_T),\underline{V}^j\in\Vc_0$ for $j=1,\ldots,d$, $\underline{\xi}\leq\xi$, $d\underline{V}\leq dV$ and $(\underline{Y},\underline{M},\underline{A})$ is a solution of $\urs(\underline{\xi},f+d\underline{V},U).$
\end{defin}

Apart from (A1)--(A5), we will also need the following assumption:
\begin{itemize}
\item[(A6)] There exists a subsolution $(\underline{Y},\underline{M},\underline{K})$ of $\urs(\xi,f+dV, U)$ such that %$\underline{Y}\leq U$ and
\[
\sum_{j=1}^d E\int_0^T \left(|f^j(r,U_r;\underline{Y}^j_r)|+ |f^j(r,\underline{Y}_r)|\right)\,dr <\infty.
\]
\end{itemize}

\begin{rem}
Assume that  (A1)--(A5) are satisfied. We define  $\underline{f}=(\underline{f}^1,\underline{f}^2\ldots,\underline{f}^d)$  by
\[
\underline{f}^j(t,c)=\inf_{y\in\R^d} f^j(t,y;c),\qquad t\in[0,T],c\in\R, j=1,\ldots,d,
\]
and assume that $\underline{f}$ satisfies  (A4) and  (A5). If there exists a semimartingale $S\leq U$ which is a difference of supermartingales of class D and such that
\[
\sum_{j=1}^d E\Big( \int_0^T |f^j(r,U_r,S^j_r)| \,dr + \int_0^T |\underline{f}^j(r,S^j_r)|\,dr \Big)<\infty,
\]
then (A6) is satisfied with $(\underline{Y},\underline{M},\underline{A})
=\{(\underline{Y}^j,\underline{M}^j,\underline{A}^j)\}_{j=1,\ldots, d}$, where $(\underline{Y}^j,\underline{M}^j,\underline{A}^j)$ is a~solution of $\urs(\xi^j, \underline{f}^j+dV^j,U^j)$ for $j=1,\ldots,d$.
\end{rem}

\begin{rem}\label{rem.2.0}
Let $(\underline{Y},\underline{M},\underline{A})$ be a subsolution satisfying (A6). Then $\underline{Y}$ is a difference of supermartingales of class D. Indeed, since $E\int_0^T|f^j(r,\underline{Y}_r)|\,dr<\infty$ and
\[
\underline{Y}^j_t-\underline{Y}^j_0 + \int_0^t f^j(r,\underline{Y}_r)\,dr + \underline{V}^j_t = \underline{M}^j_t+\underline{A}^j_t,\qquad t\in[0,T]
\]
for $j=1,\ldots,d$, the process $\underline{M}+\underline{A}$ is a submartingale of class D. Therefore, by the Doob-Meyer decomposition theorem (see \cite[Theorem 2]{Meyer1962} or \cite[Theorem III.11]{Protter}), $\underline{M}$ is a uniformly integrable martingale, and, as a consequence, $E\overline{A}_T<\infty$.
Write $\underline{C}_t=-\int_0^t f(r,\underline{Y}_r)\,dr - \underline{V}_t + \underline{A}_t$. Then $\underline{Y}_t=\underline{Y}_0 + \underline{C}_t + \underline{M}_t$, $t\in[0,T]$, and $\underline{C}\in\Vc_0.$
Let $\underline{C}=\underline{C}^+-\underline{C}^-$ be the Jordan decomposition of $\underline{C}$. Then
$\underline{C}^+,\underline{C}^-$ are increasing processes such that $\underline{C}^+,\underline{C}^-\in\Vc_0$. Therefore $S^1=\underline{M}-\underline{C}^-$ and $S^2=\underline{Y}_0-\underline{C}^+$ are supermartingales of class D. Of course, $\underline{Y}=S^1-S^2$.
\end{rem}

Let $H=(H^1,H^2,\ldots,H^d):[0,T]\times\R^d\to\R^d.$ We  will use
the notation $H_t(x)=H(t,x)$, $H^j_t(x)=H^j(t,x)$ and we adopt the
convention that for a given stochastic process $X$,
$H_{t-}(X_{t-})=\lim_{s\nearrow t} H_s(X_s)$.

In what follows we assume that $H_T(\xi)\leq\xi\leq U_T$.

\begin{defin}
We say that a quadruple
$(Y,M,K,A)=\{(Y^j,M^j,K^j,A^j)\}_{j=1,\ldots, d}$ of adapted
$\R^d$--valued \cadlag processes is a solution of a system of
BSDEs with generator $f$, terminal condition $\xi$, oblique
reflection driven by a function $H$ and upper barrier $U$ (we
write $\ors(\xi,f+dV,H,U)$ for short) if
\begin{itemize}
\item[(i)] $Y^j$ is of class D, $K^j,A^j\in\Vp_0^+$, $M^j\in\M_{loc}$ with $M^j_0=0$ for $j=1,\ldots,d$,
\item[(ii)] $H_t(Y_t)\leq Y_t\leq U_t$ for $t\in[0,T]$,
\item[(iii)] $\int_0^T (Y^j_{r-}-H^j_{r-}(Y_{r-}))\,dK^j_t=\int_0^T (U^j_{r-}-Y^j_{r-})\,dA^j_r=0$ for $j=1,\ldots,d$,
\item[(iv)] $Y^j_t=\xi^j+\int_t^T f^j(r,Y_r)\,dr+\int_t^TdV^j_r+\int_t^T dK^j_r- \int_t^T dA^j_r - \int_t^T dM^j_r$ for all $t\in[0,T]$ and $j=1,\ldots,d$.
\end{itemize}
\end{defin}
\begin{rem}\label{rem.3.2}
If $(Y,M,K,A)$ is a solution of $\ors(\xi,f+dV,H,U)$, then for each $j=1,\ldots,d$ the quadruple $(Y^j,M^j,K^j,A^j)$ is a solution of the one--dimensional equation $\rs(\xi^j,f^j(\cdot,Y;\cdot)+dV^j,H^j(Y),U^j)$.
\end{rem}

%We need to assume some regularity of oblique function $L$.
To prove the existence of a solution we will also need the
following assumption.
\begin{itemize}
\item[(A7)] $(t,y)\mapsto H_t(y)$ is continuous and $H(t,y)\leq H(t,y')$ for all $t\in[0,T]$ and $y,y'\in\R^d$, $y\leq y'.$
\end{itemize}

Note that from (A6) it follows in particular that the process $H(Y)$ is c\`adl\`ag.

\begin{ex}
An important example of $H$ is given by \eqref{eq.1.2}. It is easy to check that such $H$ satisfies (A6) (see \cite{HamZhang2010,HuTang2010}).
\end{ex}

To establish uniqueness, some additional assumptions on  $\{c^{j,k}\}$ will be needed  (e.g. $c^{i,j}+c^{j,k}> c^{i,j}$ if $i,j,k$ are all  different).

\begin{rem}\label{rem.3.3}
In \cite{TangZhongKoo2013} the problem  of existence of solutions
of \eqref{eq.1.1} is considered in two cases: when $U$ is a
pointwise limit of some decreasing sequence of It\^o processes
(see \cite[Definition 2.4]{TangZhongKoo2013}), and when $H$
satisfies \eqref{eq.1.2} with some constant functions
$\{c^{j,k}\}_{j,k=1,\ldots, d}$ such that $c^{j,j}=0$ for
$j=1,\ldots,d$ and $c^{j,k}>0$ for $j\neq k.$ The first assumption
on $U$  seems to be too restrictive and, in general, may be
difficult to verify. The second one is the special case of the
situation when the  processes $H^j(U)$ and $U^j$ are completely
separated for $j=1,\ldots,d$. Reflected BSDEs with usual
reflection, general filtration and completely separated barriers
are investigated in \cite{MT2016}. It is shown there that if the
barriers are locally separated, then they  locally satisfy a
Mokobodzki-type  condition (see \cite[Lemma 3.2]{MT2016}).
\end{rem}

The following example shows that for the existence of a solution
to systems of equations with upper barrier and oblique reflection
from below it is natural to assume that $H_t(U_t)\leq U_t$ for
$t\in[0,T]$.
%Observe that inequalities $L^j(Y)\leq Y^j$, $j=1,2,\ldots,d$ are a kind of regularity condition for process $Y$.
\begin{ex}
Let $d=2$, $T=2$ and
\[
H_t(y^1,y^2)=(y^2-1,y^1-1),\qquad t\in[0,2],\; y^1,y^2\in\R.
\]
Let $f\equiv 0$, $V\equiv 0$, $U_t=(2,t)$ for $t\in[0,2]$ and
$\xi=(2,2)$. Then $\ors(\xi,f,H,U)$ has no solution. Indeed,
suppose that $(Y,M,K,A)$ is a solution of $\ors(\xi,f,H,U)$. Then,
by Remark \ref{rem.3.2}, $(Y^1,M^1,K^1,A^1)$ is a solution of
$\rs(\xi^1,0,Y^2-1, 2)$ and $(Y^2,M^2,K^2,A^2)$ is a solution of
$\rs(\xi^2,0,Y^1-1, U^2)$. Since $Y^2_t-1\leq U^2_t-1<2$ for
$t\in[0,2]$, we see that $(2,0,0,0)$ is a solution of
$\rs(\xi^1,0,Y^2-1, 2)$. By uniqueness (see \cite[Corollary
3.2]{klimsiak2014}), $(2,0,0,0)=(Y^1,M^1,K^1,A^1).$ Therefore
$1=H^2_t(Y_t)\leq Y^2_t \leq U^2_t=t$, $t\in[0,2]$, which is a
contradiction.
%Since $H^j_t(Y_t)\leq Y^j_t$ for $t\in[0,2]$,
%\[
%Y^1_t-1\leq Y^2_t\leq Y^1_t+1.
%\]
\end{ex}

In the present paper we will need the following version of  the
Mokobodzki condition.

\begin{itemize}
\item[(M)] There exists a semimartingale $X$ being a
difference of two supermartingales of class D such that
$H(U)\leq X\leq U$.
\end{itemize}

\begin{tw}\label{tw.3}
Assume that \emph{(A1) --- (A7), (M)} are satisfied. Then there exists a solution $(Y,M,K,A)$ of $\ors(\xi,f+dV,H,U)$.
\end{tw}
\begin{proof}
Let $(\underline{Y},\underline{M},\underline{A})$ be  the
subsolution appearing in assumption (A6). We begin with showing
that if $\tilde{Y}$ is  a \cadlag process such that
$\underline{Y}_t\leq \tilde{Y}_t\leq U_t$ for $t\in[0,T]$, then
the semimartingale $X:=\underline{Y}^j$ and $\xi^j, V^j,\tilde {f}^j$, where  $\tilde{f}^j$ 
is defined by
$\tilde{f}^j(t,c)=f^j(t,\tilde{Y}_t;c)$, $t\in[0,T]$, $c\in\R$,
satisfy the assumptions of Proposition \ref{stw.3.1} for
$j=1,\ldots,d$. Obviously $c\mapsto\tilde{f}^j(t,c)$ is
decreasing. Since $\tilde{Y}$ has \cadlag paths, the random
variables $\underline{\tilde{Y}}=\inf_{t\in[0,T]}\tilde{Y}_t$ and
$\overline{\tilde{Y}}=\sup_{t\in[0,T]}\tilde{Y}_t$ are finite. By
(A3),
\begin{equation*}%\label{eq.3.2}
f^j(t,\underline{\tilde{Y}};c)\leq \tilde{f}^j(t,c)\leq f^j(t,\overline{\tilde{Y}};c),\qquad t\in[0,T],\; c\in\R,
\end{equation*}
so by (A5), $\int_0^T|\tilde{f}^j(r,c)|\,dr<\infty$. By (A3), we also have
\[
f^j(t,\underline{Y}_t)\leq \tilde{f}^j(t,\underline{Y}^j_t)\leq f^j(t,U_t;\underline{Y}^j_t),\qquad t\in[0,T].
\]
Therefore, by (A6), $E\int_0^T|\tilde{f}^j(r,\underline{Y}^j_r)|\,dr<\infty$. Moreover, by Remark \ref{rem.2.0}, $\underline{Y}^j$ is a difference of supermartingales of class D, so \eqref{eq.3.1} is satisfied with $S=\underline{Y}^j$. The rest of the proof we divide into three steps.

{\bf Step 1.} %By Theorem \ref{tw.2}, there exists a solution
%$(Y^{(0)},M^{(0)},A^{(0)})$ of $\urs(\xi, f+dV, U)$ such that
%$Y^{(0)}_t\geq \underline{Y}_t,$ $t\in[0,T]$.
Let $(Y^{(0)},M^{(0)},K^{(0)},A^{(0)})=(\underline{Y},\underline{M},0,\underline{A})$.
For $n\in\N$ we set
\[
(Y^{(n)},M^{(n)},K^{(n)},A^{(n)})=\{(Y^{(n),j},M^{(n),j},K^{(n),j},A^{(n),j})\}_{j=1,\ldots, d},
\]
where $(Y^{(n),j},M^{(n),j},K^{(n),j},A^{(n),j})$ is a solution of $\rs(\xi^j, f^j(\cdot,Y^{(n-1)};\cdot)+dV^j,$\break $H^j(Y^{(n-1)}),U^j)$. The solutions $(Y^{(n),j},M^{(n),j},K^{(n),j},A^{(n),j})$ exist by Proposition \ref{stw.3.1}, because we already know that \eqref{eq.3.1} is satisfied, and moreover, by (A7) and (M),
\[
H^j(Y^{(n-1)})\leq H^j(U)\leq X^j\leq U^j.
\]
By \cite[Proposition 2.1]{klimsiak2014}, $Y^{(1)}_t\geq Y^{(0)}_t$ for $t\in[0,T]$ (we consider $Y^{(1),j}$ as the first component of the solution of $\urs(\xi^j, f^j(\cdot,Y^{(0)};\cdot)+dV^j+dK^{(1),j},U^j)$). Suppose that for some $n\in\N$, $Y^{(n+1)}_t\geq Y^{(n)}_t$ for every $t\in[0,T]$. Then, by (A3) and (A7), $f^j(t,Y^{(n+1)}_t,y)\geq f^j(t,Y^{(n)}_t,y)$ and $H_t(Y^{(n+1)}_t)\geq H_t(Y^{(n)}_t)$ for $y\in\R,$ $t\in[0,T]$, so  by \cite[Proposition 3.1]{klimsiak2014}, $Y^{(n+2)}_t\geq Y^{(n+1)}_t$ for $t\in[0,T].$ Thus the sequence $\{Y^{(n)}\}$ is increasing. Since it is  also bounded from above by $U$, we  can define a process $Y=(Y^1,Y^2,\ldots,Y^d)$ by
\begin{equation}\label{eq.3.0}
Y^j_t=\lim_{n\to\infty}Y^{{(n)},j}_t=\sup_{n\in\N}Y^{{(n)},j}_t,\qquad t\in[0,T],\quad j=1,\ldots,d.
\end{equation}
Note that by (A7), $H^j(Y)\leq Y^j\leq U^j$ since $H^j(Y^{(n-1)})\leq Y^{(n),j}\leq U^j$ for every $n\in\N$.

{\bf Step 2.} We show that $Y=(Y^1,\dots,Y^d)$ is a \cadlag semimartingale of class D. To this end, we first assume additionally that
\begin{equation}\label{eq.3.MC}
\sum_{j=1}^dE\int_0^T |f^j(r,U_r;X^j_r)|\,dr <\infty.
\end{equation}
Fix $j$. By   \cite[Theorem 2.8]{Klimsiak2016}, for $n\in\N$ and
$p,q\in\N$ there exists a  process $Y^{(n),j,p,q}$  of class D and
a local martingale $M^{(n),j,p,q}$  such that $P$-a.s. we have
\begin{equation}\label{eq.3.4}
\begin{split}
Y^{(n),j,p,q}_t&=\xi^j + \int_t^T f^j(r, Y^{(n-1)}_r; Y^{(n),j,p,q}_r)\, dr\\
&\quad + \int_t^T dV^j_r + \int_t^T p(H^j_r(Y^{(n-1)}_r)-Y^{(n),j,p,q}_r)^+\,dr\\
&\quad - \int_t^Tq (Y^{(n),j,p,q}_r-U^j_r)^+\,ds - \int_t^T dM^{(n),j,p,q}_r,\quad t\in[0,T].
\end{split}
\end{equation}
%The pair $(Y^{(n),j,p,q},M^{(n),j,p,q}$)  exist by \cite[Theorem %2.8.]{Klimsiak2016} applied to the pair $(\xi^j, \tilde{f}^j+dV),$ with %$\tilde{f}^j(t,y)=f^j(t,Y^{(n-1)}_t;y)+p(H^j_t(Y^{(n-1)}_t)-y)^+-q(y-U^j_t)^+$.
By Remark \ref{rem.3.1}, there exists $C^j=C^{j,+}-C^{j,-}\in\V^1_0$ and $N^j\in\M$ with $N^j_0=0$
such that $X^j_t=X^j_0+C^j_t+N^j_t$. Since $X^j\leq U^j$, we have
\begin{align*}
X^j_t&=X^j_T+\int_t^T f^j(r,U_r;X^j_r)\,dr + \int_t^T f^{j,-}(r,U_r;X^j_r)\,dr+ \int_t^T dC^{j,-}_r\\
&\quad - \int_t^T f^{j,+}(r,U_r;X^j_r)\,dr - \int_t^T dC^{j,+}_r - \int_t^T q(X^j_r-U^j_r)^+\,dr - \int_t^T dN^j_r.
\end{align*}
Let $V^j=V^{j,+}-V^{j,-}$ be the Jordan decomposition of $V^j.$ By \cite[Theorem 2.7]{klimsiak2014} there exists a solution $(\overline{X}^{j,q},\overline{N}^{j,q})$ of the BSDE
\begin{align*}
\overline{X}^{j,q}_t&=X^j_T\vee\xi^j+\int_t^T f^j(r,U_r;\overline{X}^{j,q}_r)\,dr + \int_t^T f^{j,-}(r,U_r;X^j_r)\,dr\\
&\quad + \int_t^T dV^{j,+}_r + \int_t^T dC^{j,-}_r - \int_t^T q(\overline{X}^{j,q}_r-U^j_r)^+\,dr - \int_t^T d\overline{N}^{j,q}_r.
\end{align*}
By \cite[Proposition 2.1]{kliroz}, $\overline{X}^{j,q}\geq X^j,$ so $\overline{X}^{j,q}\geq H^j(Y^{(n-1)})$. Therefore the above equation may be rewritten in the form
\begin{equation}\label{eq.3.5}
\begin{split}
\overline{X}^{j,q}_t&=X^j_T\vee\xi^j+\int_t^T f^j(r,U_r;\overline{X}^{j,q}_r)\,dr
 + \int_t^T f^{j,-}(r,U_r;X^j_r)\,dr \\
 &\quad+ \int_t^T dV^{j,+}_r  + \int_t^T dC^{j,-}_r
+ \int_t^T p(H^j_r(Y^{(n-1)}_r)-\overline{X}^{j,q}_r)^+\,dr\\
&\quad - \int_t^T q(\overline{X}^{j,q}_r-U^j_r)^+\,dr - \int_t^T d\overline{N}^{j,q}_r.
\end{split}
\end{equation}
From (\ref{eq.3.4}), (\ref{eq.3.5}) and  \cite[Proposition 2.1]{kliroz} it follows that
\begin{equation}\label{eq.szac1}
Y^{(n),j,p,q}_t\leq \overline{X}^{j,q}_t,\qquad t\in[0,T].
\end{equation}
%Then
%\[
%q(Y^{(n),j,p,q}_s-U^j_s)^+ds\leq q(\overline{X}^{j,q}_s-U^j_s)^+ds.
%\]
Let $\tilde{f}^j=f^j(\cdot,U;\cdot)+f^{j,-}(\cdot,U;X^j)+ p(H^j(Y^{(n-1)})-\cdot)^+$. By \cite[Theorem 2.13]{klimsiak2014} there exists a unique solution  of $(\overline{X}^j,\overline{N}^j,\overline{D}^j)$ of $\urs(X^j_T\vee\xi^j, \tilde{f}^j+dV^{j,+}+dC^{j,-}, U^j)$ such that
\begin{equation}\label{eq.3.6}
E\overline{D}^j_T <\infty.
\end{equation}
Letting $q\rightarrow\infty$ in \eqref{eq.3.5} and applying
\cite[Theorem 2.13]{klimsiak2014} shows that
\begin{equation}\label{eq.3.05}
\overline{X}^{j,q}\searrow\overline{X}^{j}.
\end{equation}
%\[
%\overline{X}^{j,q}\searrow\overline{X}^{j}, \qquad E\int_0^Tq(\overline{X}^{j,q}_s-U^j_s)^+ds \to E\overline{D}^j_T,
%\]
Similarly, letting $q\rightarrow\infty$ in \eqref{eq.3.4} and applying
\cite[Theorem 2.13]{klimsiak2014} shows that
\begin{equation}\label{eq.3.051}
Y^{(n),j,p,q}\searrow Y^{(n),j,p},
\end{equation}
%\[
%Y^{(n),j,p,q}\searrow Y^{(n),j,p}, \qquad E\int_0^Tq(Y^{(n),j,p,q}_s-U^j_s)^+ds \to EA^{(n),j,p}_T,
%\]
where $(Y^{(n),j,p},M^{(n),j,p},A^{(n),j,p})$ is a solution of $\urs(\xi^j, \overline{f}+dV^j, U^j)$ with $\overline{f}^j=f^j(\cdot,Y^{(n-1)},\cdot)+ p(H^j(Y^{(n-1)})-\cdot)^+$. Note that $\tilde{f}^j\geq \overline{f}^j$ and $dV^{j,+}+dC^{j,-}\geq dV^j$.
By \cite[Proposition 2.14]{klimsiak2014} applied to $(\overline{X}^j,\overline{N}^j,\overline{D}^j)$ and $(Y^{(n),j,p},M^{(n),j,p},A^{(n),j,p})$, we have
\begin{equation}\label{eq.3.061}
dA^{(n),j,p}\leq d\overline{D}^j.
\end{equation}
By \cite[Theorem 3.3]{klimsiak2014},
\begin{equation}\label{eq.3.06}
Y^{(n),j,p}_t\nearrow Y^{(n),j}_t,\qquad A^{(n),j,p}_t\nearrow A^{(n),j}_t,\quad t\in[0,T]
\end{equation}
as $p\rightarrow\infty$. By \eqref{eq.3.061} and the second convergence in \eqref{eq.3.06},
\begin{equation}\label{eq.3.7}
dA^{(n),j}\leq d\overline{D}^j.
\end{equation}
Moreover, by Lemma \ref{lm.3.3}, $dA^{(n),j}\leq dA^{(n+1),j}$ for $n\in\N$. Set $A^j_t= \lim_{n\to\infty} A^{(n),j}_t$, $t\in[0,T]$. Then $A^j\in\Vp_0^+$ and, by \eqref{eq.3.6} and \eqref{eq.3.7}, $EA^j_T<\infty$.  Furthermore, $dA^{(n),j}\le dA^j $ for $n\in\N$, so
\begin{equation}\label{eq.3.8}
%\|dA^{(n),j}- dA^j\|_{TV}
\sup_{t\in[0,T]}|A^{(n),j}_t - A^j_t| = A^j_T-A^{(n),j}_T \to 0\quad P\mbox-{a.s.}
\end{equation}
%where $\|\cdot\|_{TV}$ is the total variation norm.
% Indeed, it suffices to note that
%\[
%\leq |A^{(n),j} - A^j|_T = \|dA^{(n),j} - dA^j\|_{TV}
%\]
%and apply \eqref{eq.3.8}.
By \eqref{eq.szac1}, \eqref{eq.3.05} and \eqref{eq.3.051}, $Y^{(n),j,p}\leq\overline{X}^j$ for every $p\in\N$. Therefore, by the first convergence in \eqref{eq.3.06},
\[
Y^{(n),j}\leq \overline{X}^j.
\]
Moreover, by (A2) and  (A3),
\begin{equation}\label{eq.3.13}
f^j(t, Y^{(n-1)}_t; U^{j}_t)\leq f^j(t, Y^{(n-1)}_t; Y^{(n),j}_t)\leq f^j(t, U_t; Y^{(n),j}_t),\quad t\in[0,T].
\end{equation}
By (A3) and (A4), the left-hand side of the  first inequality above increases to  $f^j(r, Y_r; U^{j}_r)$ as $n\rightarrow\infty$, whereas by (A2) and (A4) the right-hand side of the  second inequality  decreases to $f^j(r, U_r; Y^{j}_r)$. Moreover, by (A2) and (A3),
\[
\begin{split}
|f^j(t, Y^{(n-1)}_t; U^{j}_t)|&\leq |f^j(t,\overline{Y};\underline{U}^j)|+|f^j(t,\underline{Y}^{(0)};\overline{U}^j)|,\\
|f^j(t, U_t; Y^{(n),j}_t)| &\leq |f^j(t,\overline{U};\underline{Y}^{(0),j})|+|f^j(t,\underline{U};\overline{Y}^j)|,
\end{split}
\]
where $\overline{U}^j=\sup_{t\in[0,T]}U^j_t,$ $\underline{U}^j=\inf_{t\in[0,T]}U^j_t$,  $\underline{Y}^{(0),j}=\inf_{t\in[0,T]}Y^{(0),j}_t$ and $\overline{U}=(\overline{U}^1,\ldots,\overline{U}^d)$, $\underline{U}=(\underline{U}^1,\ldots,\underline{U}^d),$ and $\underline{Y}^{(0)}=(\underline{Y}^{(0),1},\ldots,\underline{Y}^{(0),d})$. Therefore, by (A5) and the dominated convergence theorem, $P$-a.s. we have
\[
\begin{split}
\int_0^T |f^j(r, Y^{(n-1)}_r; U^{j}_r)-f^j(r, Y_r; U^{j}_r)|\,dr&\to 0,\\
\int_0^T |f^j(r, U_r; Y^{(n),j}_r)-f^j(r, U_r; Y^{j}_r)|\,dr&\to 0
\end{split}
\]
as $n\rightarrow\infty$. From this and (\ref{eq.3.13}) it follows that $P$-a.s. the sequence $\{f^j(\cdot, Y^{(n-1)}; Y^{(n),j})\}$ is uniformly integrable on $[0,T]$. By \eqref{eq.3.0} and (A4),
\[
\int_0^T|f^j(r, Y^{(n-1)}_r; Y^{(n),j}_r) -f^j(r, Y_r)|\,dr \to 0\quad P\mbox{-a.s.}
\]
as $n\rightarrow\infty$. In particular, $\sup_{t\in[0,T]} \big| \int_0^tf^j(r, Y^{(n-1)}_r; Y^{(n),j}_r)\, dr - \int_0^t f^j(r, Y_r)\,dr\big|\rightarrow0$ $P$-a.s.  as $n\rightarrow\infty$.
Since $dA^{(n),j}\leq dA^j$, $|A^{(n),j}|_t\leq |A^j|_t$ for $t\in[0,T]$. Moreover, by \eqref{eq.3.13}, (A2) and (A3),
\[
\int_0^t|f^j(r, Y^{(n-1)}_r; Y^{(n),j}_r)|\, dr\leq \int_0^t \left(|f^j(r, Y^{(0)}_r; U^{j}_r)|+|f^j(r, U_r; Y^{(0),j}_r)|\right)\,dr
\]
for $t\in[0,T]$. Hence
\[
\begin{split}
\Big| \int_0^\cdot f^j(r, Y^{(n-1)}_r&; Y^{(n),j}_r)\, dr + \int_0^\cdot d(V^j-A^{(n),j})_r\Big|_t\\
& \leq \int_0^t \Big(|f^j(r, Y^{(0)}_r; U^{j}_r)| + |f^j(r, U_r; Y^{(0),j}_r)|\Big)\,dr + |V|^j_t+A^{j}_t
\end{split}
\]
for $t\in[0,T]$. Since the right--hand side of the above inequality is a \cadlag process, there exists a~sequence of stopping times $\{\tau_k\}$ such that for each $k\in\N$,
\[
\sup_{n\in\N}E\Big| \int_0^\cdot f^j(r, Y^{(n-1)}_r; Y^{(n),j}_r)\, dr + \int_0^\cdot d(V^j-A^{(n),j})_r\Big|_{\tau_k\wedge T}^2 < \infty.
\]
%Since the right--hand side of the above inequality is \cadlag, it is also locally in $L^2$. Hence $\{| \int_0^\cdot f^j(r, Y^{(n-1)}_r; Y^{(n),j}_r)\, dr + \int_0^\cdot d(V^j-A^{(n),j})_r|\}$ is locally in $L^2$.
%Moreover the variation of the process $\int_0^\cdot f^j(r, Y_r)\, dr + \int_0^\cdot d(V^j-A^{j})_r$ is locally $L^2$. Indeed, it suffices to note that the sequence $\tau_k$ defined as
%\[
%\tau_k=\inf\left\{t>0\;:\; \left|\int_0^\cdot f^j(r, Y_r)\, dr + \int_0^\cdot d(V^j-A^{j})_r\right|_t>k\right\}\wedge T
%\]
%is of stationary type, since $V^j-A^j$ is a \cadlag process. Moreover, by \eqref{eq.3.8},
%\[
%\int_0^\cdot f^j(r, Y^{(n-1)}_r; Y^{(n),j}_r)\, dr + \int_0^\cdot d(V^j-A^{(n),j})_r\to \int_0^\cdot f^j(r, Y_r)\, dr + \int_0^\cdot d(V^j-A^{j})_r
%\]
%in total variation, $P$--a.s.
%Therefore $\{|\int_0^\cdot f^j(r, Y^{(n-1)}_r; Y^{(n),j}_r)\, dr + \int_0^\cdot d(V^j-A^{(n),j})_r|\}$ is locally bounded in $L^2$, since $\int_0^\cdot f^j(r, Y^{(n-1)}_r; Y^{(n),j}_r)\, dr + \int_0^\cdot d(V^j-A^{(n),j})_r\to \int_0^\cdot f^j(r, Y_r)\, dr + \int_0^\cdot d(V^j-A^{j})_r$ uniformly $P$--a.s.
We have shown that for $j=1,\ldots,d$ the sequence $\{Y^{(n),j}\}$ satisfies the assumptions of \cite[Proposition 3.10]{Klimsiak2016} with the increasing processes $K^{(n),j}$ and the finite variation processes $\int_0^\cdot f^j(r, Y^{(n-1)}_r; Y^{(n),j}_r)\, dr + \int_0^\cdot d(V^j-A^{(n),j})_r$. Therefore, for each $j$  there exists a local martingale $M^j\in\M_{loc}$ with $M^j_0=0$ and $K^j\in\Vp_0^+$ such that
\begin{equation}\label{eq.3.091}
Y^j_t=\xi^j+\int_t^T f^j(r,Y_r)\,dr +\int_t^TdV^j_r + \int_t^T dK^j_r - \int_t^T dA^j_r - \int_t^T dM^{j}_r.
\end{equation}

Now we show how to dispense with  condition \eqref{eq.3.MC} from
the proof of \eqref{eq.3.091}. For $k\in\N$, set
\[
\sigma_k=\inf\Big\{t>0 : \sum_{j=1}^d \int_0^t |f^j(r,U_r,X^j_r)|\,dr > k \Big\}\wedge T.
\]
Observe that the sequence $\{\sigma_k\}$ is increasing  and is
a chain, i.e.
\[
P\big(\liminf_{k\to\infty}\{\sigma_k=T\}\big)=1.
\]
Moreover, \eqref{eq.3.MC} is satisfied on
$[0,\sigma_k]$, i.e.
\[
\sum_{j=1}^dE\int_0^{\sigma_k} |f^j(r,U_r;X^j_r)|\,dr <\infty.
\]
Repeating step by step the arguments from  proof of
\eqref{eq.3.091}, one can show that for $j=1,\ldots,d$ there exist
$A^{(k),j}, K^{(k),j}\in\Vp_0^+$ and $M^{(k),j}\in\M_{loc}$ with
$M^{(k),j}_0=0$ such that
\[
Y^j_{t\wedge\sigma_k}=Y^j_0 - \int_0^{t\wedge\sigma_k}f^j(r,Y_r)\,dr - V^j_{t\wedge\sigma_k} - K^{(k),j}_{t\wedge\sigma_k} + A^{(k),j}_{t\wedge\sigma_k} + M^{(k),j}_{t\wedge\sigma_k},\quad t\in[0,T]
\]
By  uniqueness of the Doob-Meyer decomposition,
\[
K^{(k+1),j}_{t\wedge\sigma_k} = K^{(k),j}_{t\wedge\sigma_k},\quad A^{(k+1),j}_{t\wedge\sigma_k}=A^{(k),j}_{t\wedge\sigma_k},\quad M^{(k+1),j}_{t\wedge\sigma_k} = M^{(k),j}_{t\wedge\sigma_k},\quad t\in[0,T].
\]
for $j=1,\ldots,d$. Therefore we can put $K^j_t=K^{(k),j}_t$, $A^j_t=A^{(k),j}_t$ for $t\in[0,\sigma_k]$ and $M^j_t=\sum_{k=1}^\infty\int_{\sigma_{k-1}\wedge t}^{\sigma_k\wedge t} dM^{(k),j}_r$, and then
\[
Y^j_t = Y^j_0 + \int_0^t f^j(r,Y^j_r)\,dr - V^j_t - K^j_t + A^j_t + M^j_t,\quad t\in[0,T],
\]
which is equivalent to \eqref{eq.3.091}.

{\bf Step 3.} We show that $A,K$ satisfy the minimality conditions. We have
\begin{equation}\label{eq.3.09}
\begin{split}
\int_0^T (U^j_{r-}-Y^{j}_{r-})\,dA^{j}_r %&= \int_0^T (U^j_{r-}-Y^{j}_{r-})\,d(A^j-A^{(n),j})_r + \int_0^T (U^j_{r-}-Y^{j}_{r-})\,dA^{(n),j}_r\\
&= \int_0^T (U^j_{r-}-Y^{j}_{r-})\,d(A^j-A^{(n),j})_r\\
&\quad + \int_0^T (U^j_{r-}-Y^{(n),j}_{r-})\,dA^{(n),j}_r + \int_0^T (Y^{(n),j}_{r-}-Y^{j}_{r-})\,dA^{(n),j}_r.
\end{split}
\end{equation}
By \eqref{eq.3.8},
\begin{equation}\label{eq.3.9}
\lim_{n\to\infty}\int_0^T (U^j_{r-}-Y^{j}_{r-})\,d(A^j-A^{(n),j})_r=0,
\end{equation}
and by the definition of a solution of RBSDE, for every $n\in\N$,
\begin{equation}\label{eq.3.10}
\int_0^T (U^j_{r-}-Y^{(n),j}_{r-})\,dA^{(n),j}_r=0.
\end{equation}
Moreover, since $Y^{(n),j},Y^{j}$ are \cadlag processes, from  \eqref{eq.3.0} it follows that  for every $n\in\N$,
\begin{equation}
\label{eq.3.11}
\int_0^T (Y^{(n),j}_{r-}-Y^{j}_{r-})\,dA^{(n),j}_r\leq 0.
\end{equation}
By \eqref{eq.3.09}--\eqref{eq.3.11}, $\int_0^T (U^j_{r-}-Y^{j}_{r-})dA^{j}_r\leq 0$. Hence $\int_0^T (U^j_{r-}-Y^{j}_{r-})dA^{j}_r= 0$ since
$U^j\geq Y^{j}$.
To show minimality of $K$, we set
\[
\tau_k=\inf\Big\{t>0: \sum_{j=1}^d \int_0^t (|f^j(r,U_r;\underline{Y}^j_r)|+ |f^j(r,\underline{Y}_r;U^j_r)|)\,dr >k \Big\}\wedge T.
\]
We are going to show that on $[0,\tau_k]$ we have
\begin{equation}\label{eq.snell}
\begin{split}
Y^j_t&=\esssup_{\tau\in\T_t} E\Big(\int_t^{\tau_k\wedge\tau}f^j(r,Y_r)\,dr\\
&\qquad + \int_t^{\tau_k\wedge\tau} d(V^j-A^j)_r + H^j_\tau(Y_\tau)\ind{\{\tau<\tau_k\}}+Y^j_{\tau_k}\ind{\{\tau\geq\tau_k\}} \big| \F_t\Big).
\end{split}
\end{equation}
%%%%%%%%% opisać dokładniej! %%%%%%%%%%%%%%
Indeed, we can regard $(Y^{(n),j}, M^{(n),j}, K^{(n),j})$ as a solution of $\lrs(\xi^j,$ $f^j(\cdot,Y^{(n-1)};\cdot)+dV^j - dA^{(n),j}, H^j(Y^{(n-1)}))$, so by Remark \ref{rem.2.1}, %\cite[Remark 3.7.]{Klimsiak2016}
\[
\begin{split}
Y^{(n),j}_t&=\esssup_{\tau\in\T_t} E\Big( \int_t^{\tau_k\wedge\tau}f^j(r,Y^{(n-1)}_r;Y^{(n),j}_r)\,dr \\
&\qquad+ \int_t^{\tau_k\wedge\tau} d(V^j-A^{(n),j})_r + H^j_\tau(Y^{(n-1)}_\tau)\ind{\{\tau<\tau_k\}}+Y^{(n),j}_{\tau_k}\ind{\{\tau\geq\tau_k\}} \big| \F_t\Big).
\end{split}
\]
By the above,  (A3) and (A7),
\begin{equation}\label{eq.snelln}
\begin{split}
Y^{(n),j}_t&\leq\esssup_{\tau\in\T_t} E\Big( \int_t^{\tau_k\wedge\tau}f^j(r,Y_r;Y^{(n),j}_r)\,dr \\
&\qquad+ \int_t^{\tau_k\wedge\tau} d(V^j-A^{(n),j})_r + H^j_\tau(Y_\tau)\ind{\{\tau<\tau_k\}}+Y^{j}_{\tau_k}\ind{\{\tau\geq\tau_k\}} \big| \F_t\Big).
\end{split}
\end{equation}
By \eqref{eq.3.8} and the fact that $E A^j_T <\infty$,
\[
\sup_{\tau\in\T_t}E|A^j_\tau-A^{(n),j}_\tau|\to 0
\]
as $n\rightarrow\infty$. By (A2), (A4) and the monotone convergence theorem,
\[
E\int_t^{\tau_k\wedge\tau}|f^j(r,Y_r;Y^{(n),j}_r)-f^j(r,Y_r)|\,dr\to 0
\]
as $n\rightarrow\infty$.
Therefore letting $n\to\infty$ in \eqref{eq.snelln} and using \cite[Lemma 3.15]{KliRzySlo} we get
\begin{equation}
\label{eq3.25}
\begin{split}
Y^j_t&\leq\esssup_{\tau\in\T_t} E\Big( \int_t^{\tau_k\wedge\tau}f^j(r,Y_r)\,dr\\
&\qquad + \int_t^{\tau_k\wedge\tau} d(V^j-A^j)_r + H^j_\tau(Y_\tau)\ind{\{\tau<\tau_k\}}+Y^j_{\tau_k}\ind{\{\tau\geq\tau_k\}} \big| \F_t\Big).
\end{split}
\end{equation}
Since $Y^j_{t\wedge\tau_k}\geq H^j_t(Y_t)\ind{\{t<\tau_k\}}+Y^j_{\tau_k}\ind{\{t\geq\tau_k\}}$ for $t\in[0,T]$, the supermartingale
\[
Y^j+ Y^j_0+\int_0^\cdot f^j(r,Y_r)\,dr+\int_0^\cdot d(V^j-A^j)_r =-K^j+ M^j
\]
dominates on $[0,\tau_k]$ the process $R$ defined by
\[
R_t=Y^j_0+\int_0^t f^j(r,Y_r)\,dr + \int_0^t d(V^j-A^j)_r + H^j_t(Y_t)\ind{\{t<\tau_k\}}+Y^j_{\tau_k}\ind{\{t\geq\tau_k\}},\quad t\in[0,T].
\]
By \cite[App.I.22 Theorem]{DellMey2}, $S$ defined by $S_t=\esssup_{\tau\in\T_t}E(R_{\tau\wedge\tau_k}|\F_t)$ is the minimal supermartingale which dominates the process $R$ on $[0,\tau_k]$. Hence
\begin{equation}\label{eq.3.18}
Y^j_t+Y^j_0+\int_0^t f^j(r,Y_r)\,dr+\int_0^t d(V^j-A^j)_r\geq S_t
\end{equation}
for $t\in[0,\tau_k]$. But
\[
\begin{split}
S_t&= \esssup_{\tau\in\T_t} E\Big( \int_0^{\tau_k\wedge\tau}f^j(r,Y_r)\,dr\\
&\qquad + \int_0^{\tau_k\wedge\tau} d(V^j-A^j)_r + H^j_\tau(Y_\tau)\ind{\{\tau<\tau_k\}}+Y^j_{\tau_k}\ind{\{\tau\geq\tau_k\}} \big| \F_t\Big)\\
&= \esssup_{\tau\in\T_t} E\Big( \int_t^{\tau_k\wedge\tau}f^j(r,Y_r)\,dr\\
&\qquad + \int_t^{\tau_k\wedge\tau} d(V^j-A^j)_r + H^j_\tau(Y_\tau)\ind{\{\tau<\tau_k\}}+Y^j_{\tau_k}\ind{\{\tau\geq\tau_k\}} \big| \F_t\Big)\\
&\qquad + \int_0^t f^j(r,Y_r)\,dr + \int_0^t d(V^j-A^j)_r.
\end{split}
\]
Combining this with \eqref{eq.3.18} we obtain
\[
\begin{split}
Y^j_t&\geq\esssup_{\tau\in\T_t} E\Big( \int_t^{\tau_k\wedge\tau}f^j(r,Y_r)\,dr\\
&\qquad + \int_t^{\tau_k\wedge\tau} d(V^j-A^j)_r +
H^j_\tau(Y_\tau)\ind{\{\tau<\tau_k\}}
+Y^j_{\tau_k}\ind{\{\tau\geq\tau_k\}} \big| \F_t\Big),
\end{split}
\]
which when combined with (\ref{eq3.25}) proves \eqref{eq.snell}.
By \eqref{eq.snell} and Remark \ref{rem.2.2}, there exist processes $\widetilde{K}^j,\widetilde{M}^j$ such that the triple $(Y^j_{\cdot\wedge\tau_k},\widetilde{M}^j_{\cdot\wedge\tau_k},\widetilde{K}^j_{\cdot\wedge\tau_k})$ is a solution of $\lrs(Y^j_{\tau_k}, f^{\tau_k} +dV^j_{\cdot\wedge\tau_k} - dA^{j}_{\cdot\wedge\tau_k}, H^j_{\cdot\wedge\tau_k}(Y_{\cdot\wedge\tau_k}))$ with $f^{\tau_k}=f^j(\cdot,Y;\cdot)\ind{[0,\tau_k]}$. In particular,
\[
\int_0^{\tau_k}(Y^j_{r-}-H^j_{r-}(Y_{r-}))\,d\widetilde{K}^j_r=0.
\]
By \cite[Corollary 2.2.]{klimsiak2014}, \eqref{eq.3.091} and uniqueness of the Doob--Meyer decomposition, $\widetilde{K}^j_{\cdot\wedge\tau_k}=K^j_{\cdot\wedge\tau_k}$. Consequently,
\[
\int_0^{\tau_k}(Y^j_{r-}-H^j_{r-}(Y_{r-}))\,d{K}^j_r=0.
\]
Letting $k\to\infty$ gives the minimality condition for $K$.
\end{proof}

\begin{rem}
Thanks to \eqref{eq.3.MC}, $EA_T$ is always finite (see Step 2 of the proof of Theorem \ref{tw.3}). Moreover, if
\[
E\sum_{j=1}^d \int_0^t (|f^j(r,U_r;\underline{Y}^j_r)|+ |f^j(r,\underline{Y}_r;U^j_r)|)\,dr<\infty,
\]
then also $E\sum_{j=1}^d \int_0^T |f^j(r,Y^j_r)|\,dr<\infty$, and, by the Doob-Meyer decomposition, $E K^j_T<\infty$ and $M^j$ is an uniformly integrable martingale for $j=1,\ldots,d$.
\end{rem}
\medskip
%\hrule
%The proof of Theorem \ref{tw.3} gives an interesting example of one-dimensional reflected BSDE. Under (M), the quadruple $(Y^j,M^j,K^j,A^j)$ is a solution of $\rs(\xi^j,f^j(\cdot,Y;\cdot)+dV^j,H^j(Y),U^j)$ such that $EA^j_T$ needs to be finite, while $EK^j_T$ does not. Nevertheless, the question about an example of data, for which $EK^j_T=\infty$ remains open.
%\hrule
In the rest of this section we assume that $H$ has the form
\begin{equation}%\label{eq.3.19}
H^j_t(y) = \max_{k\neq j} h_{j,k}(t,y^k),
\end{equation}
where $\{h_{j,k}\}_{j,k=1,\ldots, d}$ are continuous functions such that $h_{j,k}(t,c)\leq c$ for $c\in\R$. Such a form of $H$ appears in applications to the switching problem (see Section \ref{sec4}). The difference $c-h_{j,k}(t, c)$ can be viewed as  a cost of switching from  mode $j$ to mode $k$ in time $t$.

In addition to (A7), we  will need the following condition
considered in \cite{HamZhang2010}:
\begin{itemize}
\item[(A8)] there are no $y_1,y_2,\ldots,y_k\in\R$ and $j_1,j_2,\ldots,j_k\in\{1,\ldots,d\}$ such that for some $t\in[0,T]$,
\begin{equation}
\label{eq3.28}
y_1=h_{j_1,j_2}(t,y_2),\quad y_2=h_{j_2,j_3}(t,y_3),\ldots, y_k=h_{j_k,j_1}(t,y_1).
\end{equation}
\end{itemize}
%Condition (A8) is a kind of the triangle inequality.
\begin{ex}
Assume $H$ has the form \eqref{eq.1.2} %\eqref{eq.3.19} with $h_{j,k}$ given by $h_{j,k}(t,y^k)=y^k-c^{j,k}(t)$ for $j,k=1,2,\ldots,d$
with some positive continuous processes $\{c^{j,k}\}$ such that
\begin{equation}\label{eq.3.20}
c^{i,j}(t)+c^{j,k}(t)>c^{i,k}(t)
\end{equation}
for all $t\in[0,T]$ and $i,j,k=1,\ldots,d$ such that $i,j,k$ are all different. Then (A8) is satisfied. Indeed, striving for a contradiction, suppose  that there exist $y_1,y_2,\ldots,y_k\in\R$ and $j_1,j_2,\ldots,j_k\in\{1,\ldots,d\}$ such that (\ref{eq3.28}) is satisfied for some  $t\in[0,T]$. Then
\[
y_1= y_1- (c^{j_1,j_2}(t)+c^{j_2,j_3}(t) +\ldots+c^{j_{k-1},j_k}(t)+c^{j_k,j_1}(t)).
\]
On the other hand,  by \eqref{eq.3.20},
\[
c^{j_1,j_2}(t)+c^{j_2,j_3}(t)+\ldots+c^{j_{k-1},j_k}(t)+c^{j_k,j_1}(t)
>c^{j_1,j_k}(t)+c^{j_k,j_1}(t)>0.
\]
\end{ex}

Assume that $\mathbb{F}$ is quasi-left continuous and $V$ is continuous. If $\Delta K^j_\tau >0$, then $Y^j_{\tau-}=H^j_{\tau-}(Y_{\tau-})$ and $\Delta Y^j_\tau = -\Delta K^j_\tau$ for $\tau\in\Tp$. In \cite[Remark 3.14]{Klimsiak2016} (see also Step 4 of the proof of \cite[Theorem 3.2]{HamZhang2010}) it is proved that in case of equations with no upper barrier condition  (A8) ensures continuity of $K$. It turns out that this is also true for equations with upper barrier.

\begin{stw}\label{lm.Kcont}
Assume that $\mathbb{F}$ is quasi-left continuous, $V$ has no predictable jumps and {\em (A8)} is satisfied. Let $(Y,M,K,A)$ be a solution of $\ors(\xi,f+dV,H,U)$. Then $K$ is continuous.
\end{stw}
\begin{proof}
Since $K$ is predictable, it is enough to show that $\Delta K_\tau=0$ for $\tau\in\Tp.$ Aiming  for a contradiction, suppose that $\Delta K^{j_1}_\tau>0$ for some $\tau\in\Tp$  and $j_1\in\{1,\ldots,d\}$.
Then  from the minimality condition for
$K^{j_1}$ it follows that $Y^{j_1}_{\tau-}= H^{j_1}_{\tau-}(Y_{\tau-})$. Consequently, by the definition of $H^{j_1}$, there exists  $j_2\neq j_1$ such that $Y^{j_1}_{\tau-}=h_{j_1,j_2}(\tau, Y^{j_2}_{\tau-})$. Hence
\[
h_{j_1,j_2}(\tau, Y^{j_2}_{\tau-})=Y^{j_1}_{\tau-} > Y^{j_1}_{\tau} \geq H^{j_1}_{\tau}(Y_{\tau})\geq h_{j_1,j_2}(\tau, Y^{j_2}_{\tau}).
\]
By the above and the minimality condition for $K^{j_2}$ we have $\Delta Y^{j_2}_\tau<0$. Since $\mathbb{F}$ is quasi-left continuous and $\tau\in\Tp$, $\Delta Y^{j_2}_\tau=\Delta A^{j_2}_\tau - \Delta K^{j_2}_\tau$. But $\Delta Y^{j_2}_\tau < 0$ and $\Delta A^{j_2}_\tau\geq 0$, so $\Delta K^{j_2}_\tau >0$.
Repeating this argument, for each $k\in\N$ we can find $j_k\in\{1,2\ldots,d\}$ such that $\Delta K^{j_k}_{\tau}>0$. Since the indices $j_k$ take values in the finite set $\{1,\dots,d\}$, without loss of generality we can assume that $j_k=j_1$ for some $k\in\N.$ We then have
\[
Y^{j_1}_{\tau-}=h_{j_1,j_2}(\tau, Y^{j_2}_{\tau-}),\quad Y^{j_2}_{\tau-}=h_{j_2,j_3}(\tau, Y^{j_3}_{\tau-}),\ldots, Y^{j_k}_{\tau-}=h_{j_k,j_1}(\tau, Y^{j_1}_{\tau-}),
\]
which is in contradiction with  (A8).
\end{proof}

\section{Switching problem and uniqueness  of solutions}
\label{sec4}

In this section we assume that  $f^j(t,y) \equiv f^j(t,y^j)$,
$y\in\R^d$, and that $H$ has the form  \eqref{eq.1.2}, i.e.
\[
H^j_t(y)=\max_{k\neq j}h_{j,k}(t,y^k)\quad\mbox{with
}h_{j,k}(t,y^k)=y^k-c^{j,k}(t).
\]
Let $\{\theta_n\}$ be an increasing
sequence of stopping times with values in $[0,T]$ such that
${N=\inf\{ n\ge1:\theta_n=T\} <\infty}$ a.s., and let
$\{\alpha_n\}$ be a sequence of random variables with values in
$\{1,\ldots,d\}$ such that $\alpha_n$ is a
$\F_{\theta_n}$-measurable for each $n\in\N$. Observe that the set
$\{N=n\}$ is $\F_{\theta_n}$-measurable for each $n\in\N$. Recall
that an admissible switching control process $a$ determined by
$\{\theta_n\}$ and $\{\alpha_n\}$ is defined by
\[
a(t) = \sum_{n=0}^{N-1}\alpha_n\ind{[\theta_n,\theta_{n+1})}(t) +
\alpha_N\ind{\{\theta_N\}}(t), \quad t\in[\theta_0,T].
\]
Note that the control $a$ is defined for some given initial data $(\alpha_0,\theta_0)$. For simplicity, we omit this dependence in our notation. In what follows we denote by $\A^j_t$ the set of all  admissible
switching control processes with the initial data
$(\alpha_0,\theta_0)=(j,t).$

Below we generalize the profitability model \eqref{eq.1.4} described
in the introduction. Set $\Fb=\Fb^X$. Assume that the power station
can  produce electricity in one of $d$ modes and that the dynamics of
the income in $j$-th mode is driven by $f^j$ for $j=1\ldots,d$
(for simplicity, in our notation  we omit the dependence of $f^j$ on $X$). For a switching strategy $a\in\A^j_t$,  the profit of
the power station is now given by the first component of a solution
$(R^{(a)}, M^{(a)},D^{(a)})$ of the reflected BSDE of the form
\begin{equation}\label{eq.switching}
\begin{cases}
R^{(a)}_s= \xi^{a(T)}+ \int_s^T f^{a(r)}(r,R^{(a)}_r)\,dr
+ \int_s^T dV^{a(r)}_r\smallskip\\
\qquad\quad-\sum_{n=1}^{N}
c^{\alpha_{n-1}, \alpha_n}(\theta_n)\ind{(s,T]}(\theta_n)
- \int_s^T dD^{(a)}_r - \int_s^T dM^{(a)}_r,\smallskip\\
R^{(a)}_s\leq U^{a(s)}_s,\qquad s\in[t,T],\smallskip\\
\int_t^T (U^{a(r-)}_{r-} - R^{(a)}_{r-})\, dD^{(a)}_r=0.
\end{cases}
\end{equation}
The processes $V^j$ in \eqref{eq.switching} allow us to take into account some external factors which are not absolutely continuous with respect to Lebesgue measure (for instance, $V^j$ can be a jump process modelling some extra costs or incomes). Our goal is to find the maximal profitability of the station, i.e. for each $j\in\{1,\dots,d\}$ we want find the quantity
\[
\esssup_{a\in\A^j_t} R^{(a)}_t,\qquad t\in[0,T].
\]
The second goal is to find the optimal switching strategy, i.e. for each $j$ and $t$ we want  find $\widehat a\in\A^j_t$ such that
\[
R^{(\widehat a)}_t=\esssup_{a\in\A^j_t} R^{(a)}_t.
\]
Before solving these problems, some  comments on the existence of a solution of \eqref{eq.switching} are in order.
\begin{stw}\label{prop.4.3}
Let $a\in\A^j_t$. If {\em (A1), (A2), (A4), (A5)} are satisfied, $U^-$ is of class {\em D} and
\begin{equation}\label{eq.4.20}
E\sum_{n=1}^N|c^{\alpha_{n-1}, \alpha_n}(\theta_n)|<\infty,
\end{equation}
then there exists a unique solution of \eqref{eq.switching}.
\end{stw}
\begin{proof}
Follows from \cite[Theorem 2.3(ii)]{MT2016}.
\end{proof}
Note that if (M) is satisfied, then $U^-$ is of class D. Later we will show, that for some $a\in\A_t^j$, there exists a solution of \eqref{eq.switching} even if we do not assume \eqref{eq.4.20} (see Theorem \ref{prop.4.4}).

\begin{lm}\label{lm.4.2}
Let $M=(M^1,\ldots,M^d)$ be an $\R^d$-valued local $\mathbb{F}$-martingale  and
$a\in \A^j_t$. Then the process
\[
M^{(a)}_s=\sum_{n=0}^{N-1} \sum_{j=1}^d\ind{\{\alpha_n=j\}}\int_{\theta_n\wedge s}^{\theta_{n+1}\wedge s}dM^{j}_r, \qquad s\in[0,T],
\]
is also a local $\mathbb{F}$-martingale.
\end{lm}
\begin{proof}
Let $\{\tau_m\}$ be a fundamental sequence for $M$.  We
will show that for all $0\leq u\leq s\leq T$ and $m\in\N$,
\begin{equation}
\label{eq4.12}
E\Big(M^{(a)}_{s\wedge\tau_m} - M^{(a)}_{u\wedge\tau_m}\Big| \F_u\Big)
=0.
\end{equation}
Write $B^j_n=\{\alpha_n=j\}$. Then
\begin{equation}\label{eq.4.3}
M^{(a)}_{s\wedge\tau_m} - M^{(a)}_{u\wedge\tau_m} = \sum_{n=0}^{N-1} \sum_{j=1}^d\ind{B^j_n}\int_{(\theta_n\vee u)\wedge s\wedge\tau_m}^{\theta_{n+1}\wedge s\wedge\tau_m}dM^{j}_r.
\end{equation}
Since the set $B^j_n$ is $\F_{\theta_n\vee u}$-measurable, we have
\begin{equation}\label{eq.4.2}
\begin{split}
E\Big(\sum_{j=1}^d\ind{B^j_n}\int_{(\theta_n\vee u)\wedge s\wedge \tau_m}^{\theta_{n+1}\wedge s\wedge\tau_m}dM^{j}_r \Big| \F_u\Big) &= E\Big(E\Big(\sum_{j=1}^d\ind{B^j_n}\int_{(\theta_n\vee u)\wedge s\wedge\tau_m}^{\theta_{n+1}\wedge s\wedge\tau_m} dM^{j}_r \Big| \F_{\theta_n\vee u}\Big) \Big| \F_u\Big)\\
&= E\Big(\sum_{j=1}^d\ind{B^j_n} E\Big(\int_{(\theta_n\vee u)\wedge s\wedge\tau_m}^{\theta_{n+1}\wedge s\wedge\tau_m} dM^{j}_r \Big| \F_{\theta_n\vee u}\Big) \Big| \F_u\Big)\\
&=0\quad P\text{-a.s.},
\end{split}
\end{equation}
where the last equality follows from the fact that $M^j_{\cdot\wedge\tau_m}$ is a true martingale for all $j=1,\ldots,d$ and $m\in\N.$
Therefore  \eqref{eq4.12} follows from \eqref{eq.4.3}, \eqref{eq.4.2} and linearity of the conditional expectation.
\end{proof}

%\begin{rem}\label{rem.4.1}
%Note that, thanks to \eqref{eq.4.5}, we can modify the result of Theorem \ref{prop.4.4} as follows.  By Theorem \ref{prop.4.4}, $\widehat{a}\in\widetilde{\A}^j_t$ (a solution of \eqref{eq.switching} is given by \eqref{eq.4.4}). Moreover,
%\[
%Y^j_t=\esssup_{a\in\widetilde{\A}^j_t} R^{(a)}_t,\qquad t\in[0,T],
%\]
%and, by Theorem \ref{tw.3} and \eqref{eq.4.10}, $\widetilde{\A}^j_t$ is always nonempty for $j=1,\ldots,d$ and $t\in[0,T]$.
%\end{rem}

Let $\widetilde{\A}^j_t$ denote the set of all admissible controls $a\in\A^j_t$ for which there exists a solution of \eqref{eq.switching}.

\begin{rem}\label{rem.4.3}
Let $a\in\widetilde{A}^j_t$ and (A2) be satisfied. Then from \cite[Corollary 2.2]{klimsiak2014} it follows that the solution of \eqref{eq.switching} is unique.
\end{rem}

\begin{tw}\label{prop.4.4}
Assume that $\mathbb{F}$ is quasi-left continuous, assumptions
{\em (A1)--(A8)}, {\em (M)} are satisfied and  $V$
has no predictable jumps.
\begin{enumerate}
\item[\rm (i)] $\widetilde\A^j_t\neq\emptyset$ for all $t\in[0,T]$ and $j=1,2,\ldots,d.$
\item[\rm(ii)]
If $(Y,M,K,A)$ is a solution of $\ors(\xi,f+dV,H,U)$, then
\[
Y^j_t=\esssup_{a\in\widetilde\A^j_t} R^{(a)}_t,\qquad t\in[0,T],
\]
where $(R^{(a)},M^{(a)},D^{(a)})$ denotes a solution of the reflected BSDE \eqref{eq.switching}.
\item[\rm(iii)] Fix $t\in[0,T]$ and $j\in\{1,\ldots,d\}.$ Define $\widehat a$ by
\[
\widehat a (s)=\sum^{\infty}_{n=0}\widehat\alpha_n\ind{[\widehat{\theta}_{n},\widehat{\theta}_{n+1})}(s),\qquad s\in[t,T],
\]
where  $\widehat{\theta}_0=t, \widehat{\alpha}_0=j$, and for
$n\ge1$ we set
\[
\widehat{\theta}_n = \inf\Big\{s > \widehat{\theta}_{n-1}:
Y^{\widehat{\alpha}_{n-1}}_s =\max_{k\neq \widehat{\alpha}_{n-1}}
h_{\widehat{\alpha}_{n-1},k}(s,Y^{\widehat{\alpha}_{n-1}}_s)
\Big\}\wedge T,
\]
where $\widehat{\alpha}_n$ is the smallest index such that
\[
Y^{\widehat{\alpha}_{n-1}}_{\widehat{\theta}_n}=
h_{\widehat{\alpha}_{n-1},\widehat{\alpha}_n}\left(\widehat{\theta}_n,Y^{\widehat{\alpha}_{n-1}}_{\widehat{\theta}_n}\right),
\]
if $\widehat{\theta}_n<T$, and $\widehat{\alpha}_n$ is an
arbitrary index  otherwise. Then %$\widehat a$ is an optimal switching
%control, i.e.
$\widehat a\in\A^j_t$ and $Y^j_t=R^{(\widehat{a})}_t.$
\end{enumerate}
\end{tw}
\begin{proof}
Fix $j\in\{1,2,\ldots,d\}$ and $t\in[0,T]$, and suppose that there exists some $a\in\widetilde\A^j_t$. For $i=1,\dots, d$ we have
\begin{equation*}
\begin{split}
Y^{i}_{(s\vee\theta_n)\wedge\theta_{n+1}}&= Y^{i}_{\theta_{n+1}}+ \int_{(s\vee\theta_n)\wedge\theta_{n+1}}^{\theta_{n+1}} f^{i}(r,Y^{i}_r)\,dr + \int_{(s\vee\theta_n)\wedge\theta_{n+1}}^{\theta_{n+1}} dV^{i}_r \\
&\quad + \int_{(s\vee\theta_n)\wedge\theta_{n+1}}^{\theta_{n+1}} dK^{i}_r -
\int_{(s\vee\theta_n)\wedge\theta_{n+1}}^{\theta_{n+1}} dA^{i}_r - \int_{(s\vee\theta_n)\wedge\theta_{n+1}}^{\theta_{n+1}}
dM^{i}_r,\qquad s\in[t,T],
\end{split}
\end{equation*}
Hence, for every $s\in[t,T]$,
\begin{equation}
\label{eq.4.04}
\begin{split}
Y^{\alpha_n}_{(s\vee\theta_n)\wedge\theta_{n+1}}&= Y^{\alpha_n}_{\theta_{n+1}}+ \int_{(s\vee\theta_n)\wedge\theta_{n+1}}^{\theta_{n+1}} f^{\alpha_n}(r,Y^{\alpha_n}_r)\,dr + \int_{(s\vee\theta_n)\wedge\theta_{n+1}}^{\theta_{n+1}} dV^{\alpha_n}_r \\
&\quad + \int_{(s\vee\theta_n)\wedge\theta_{n+1}}^{\theta_{n+1}} dK^{\alpha_n}_r -
\int_{(s\vee\theta_n)\wedge\theta_{n+1}}^{\theta_{n+1}} dA^{\alpha_n}_r\\
&\quad - \sum_{i=1}^d\ind{B^i_n}\int_{(s\vee\theta_n)\wedge\theta_{n+1}}^{\theta_{n+1}}
dM^{i}_r,
\end{split}
\end{equation}
where $B^i_n=\{\alpha_n=i\}.$ For $s\in[t,T]$ we set
\[
Y^{(a)}_s=Y^{a(s)}_s,\qquad V^{(a)}_s=\sum_{n=0}^{N-1} \int_{\theta_n\wedge s}^{\theta_{n+1}\wedge s} dV^{\alpha_n}_r, \qquad M^{(a)}_s=\sum_{n=0}^{N-1}\sum_{i=1}^d\ind{B^i_n}\int_{\theta_n\wedge s}^{\theta_{n+1}\wedge s}dM^{i}_r,
\]
\[
K^{(a)}_s=\sum_{n=0}^{N-1} \int_{\theta_n\wedge s}^{\theta_{n+1}\wedge s} dK^{\alpha_n}_r,\qquad A^{(a)}_s=\sum_{n=0}^{N-1} \int_{\theta_n\wedge s}^{\theta_{n+1}\wedge s} dA^{\alpha_n}_r.
\]
Note that $\Delta A^{(a)}_{\theta_n}=\Delta A^{\alpha_{n-1}}_{\theta_n}$ for $n\in\N$ and, by Lemma \ref{lm.4.2}, $M^{(a)}$ is a local martingale. Using  \eqref{eq.4.04} and the fact that $(Y,M,K,A)$ is a solution of $\ors(\xi,f+dV,H,U)$ one can check that %, \eqref{eq.4.041}
%and Proposition \ref{lm.Kcont}
\begin{equation}\label{eq.4.4}
\begin{cases}
Y^{(a)}_s= \xi^{a(T)}+ \int_s^T f^{a(r)}(r,Y^{(a)}_r)\,dr
+ \int_s^T dV^{(a)}_r + \int_s^T dK^{(a)}_r
- \int_s^T dA^{(a)}_r\smallskip\\
\qquad\qquad -\sum_{n=1}^{N} \left(Y^{\alpha_n}_{\theta_n}
- Y^{\alpha_{n-1}}_{\theta_n} \right)\ind{(s,T]}(\theta_n)
- \int_s^T dM^{(a)}_r,\smallskip\\
Y^{(a)}_s\leq U^{(a)}_s,\qquad s\in[t,T],\smallskip \\
\int_t^T (U^{(a)}_{r-} - Y^{(a)}_{r-})\, dA^{(a)}_r=0,
\end{cases}
\end{equation}
where $U^{(a)}_s=U^{a(s)}_s$ for $s\in[t,T]$.  Hence $(Y^{(a)},
M^{(a)}, A^{(a)})$ is a~solution on $[t,T]$ of the equation
$\urs(\xi^{a(T)}, f^{(a)}+dV^{(a)}-dV^{Y^a}+dK^{(a)}, U^{(a)})$,
with $f^{(a)}(s,y)=f^{a(s)}(s,y)$ for $s\in[t,T],$ $y\in\R$ and
\[
V^{Y^a}_s = \sum_{n=1}^{N} \left(Y^{\alpha_n}_{\theta_n} -
Y^{\alpha_{n-1}}_{\theta_n} \right)\ind{(0,s]}(\theta_n), \quad
s\in[t,T].
\]
%In particular, since $Y^{(a)}_t=Y^j_t$ for $a\in\A^j_t$, we have
%\begin{equation}\label{eq.4.5}
%Y^j_t\geq R^{(a)}_t,\quad a\in\A^j_t.
%\end{equation}
We now show that $\widehat{a}\in\widetilde\A^j_t.$ By the
construction, $(\widehat\alpha_0,\widehat\theta_0)=(j,t)$, so we have to
show that $P(B)=0$, where
$B=\bigcap_{j=1}^\infty\{\widehat{\theta}_n<T\}.$ Observe that
$B\in\F_T$ and for $n=1,2,\ldots$ we have
\begin{equation}\label{eq.4.7}
Y^{\widehat{\alpha}_{n-1}}_{\widehat{\theta}_n} =
h_{\widehat{\alpha}_{n-1},\widehat{\alpha}_n}
\left(\widehat{\theta}_n,
Y^{\widehat{\alpha}_{n}}_{\widehat{\theta}_n}\right)
%Y^{\widehat{\alpha}_{n}}_{\widehat{\theta}_{n+1}} &= h_{\widehat{\alpha}_{n},\widehat{\alpha}_{n+1}}\left(\widehat{\theta}_{n+1}, Y^{\widehat{\alpha}_{n+1}}_{\widehat{\theta}_{n+1}}\right)
\end{equation}
on $B$. Since $\{\widehat{\alpha}_n\}$ takes values in
$\{1,\ldots,d\},$ there exists $k\geq 2$ such that
$\widehat{\alpha}_{k-1}=\widehat{\alpha}_0$. Since
$\{(\widehat{\alpha}_{n},\widehat{\alpha}_{n+1},
\ldots,\widehat{\alpha}_{n+k-1})\}_{n=1,\ldots}$  takes values
in $\{1,\ldots,d\}^k$, there is a vector of indexes
$(j_1,j_2,\ldots, j_{k})$ and a subsequence $\{n_m\}$ such that
\[
(\widehat{\alpha}_{n_m},\widehat{\alpha}_{n_m+1},\ldots,\widehat{\alpha}_{n_m+k-1})=(j_1,j_2,\ldots, j_{k}).
\]
The sequence $\{\widehat{\theta}_n\}$ is non-decreasing  and
bounded by $T$, so it is convergent. Assume that
$\widehat{\theta}_n\nearrow\widehat{\theta}_\infty.$ Passing to
the limit in  \eqref{eq.4.7} along this subsequence gives
\begin{equation*}%\label{eq.4.8}
\begin{split}
Y^{j_1}_{\widehat{\theta}_\infty-} &= h_{j_1,j_2}
\left(\widehat{\theta}_\infty,
Y^{j_2}_{\widehat{\theta}_\infty-}\right),\\
Y^{j_2}_{\widehat{\theta}_\infty-} &= h_{j_2,j_3}
\left(\widehat{\theta}_\infty,
Y^{j_3}_{\widehat{\theta}_\infty-}\right),\\
& \;\;\vdots\\
Y^{j_{k}}_{\widehat{\theta}_\infty-} &
= h_{j_{k},j_1}\left(\widehat{\theta}_\infty,
Y^{j_1}_{\widehat{\theta}_\infty-}\right)
\end{split}
\end{equation*}
on the set $B$. By (A8), this implies that $P(B)=0.$ Thus
$\widehat{a}\in\A^j_t.$ By \eqref{eq.4.7},
\begin{equation}\label{eq.4.9}
Y^{\widehat\alpha_{n-1}}_{\widehat\theta_n} = Y^{\widehat\alpha_n}_{\widehat\theta_n}-c^{\widehat\alpha_{n-1},\widehat\alpha_n}(\widehat\theta_n).
\end{equation}
%%%%%%%%%%%%%%% Szkic %%%%%%%%%%%%%%%%%%%%%%%%%%%%%%%%%
By Proposition \ref{lm.Kcont}, $K$ is continuous. Hence $\int_t^T
dK^{(\widehat{a})}_r=0$ by the construction of $\widehat{a}.$
Therefore from  \eqref{eq.4.4} and \eqref{eq.4.9} it follows that
the triple   $(Y^{(\widehat a)},M^{(\widehat a)},A^{(\widehat a)})$ is a
solution of \eqref{eq.switching}. This  proves (i). %Furthermore, by \eqref{eq.4.5},
%\begin{equation}\label{eq.4.010}
%Y^j_t\geq R^{(a)}_t,\quad a\in\widetilde\A^j_t
%\end{equation}
%since $Y^{(a)}_t=Y^j_t$ for $a\in\widetilde\A^j_t.$
By Remark \ref{rem.4.3}, the solution of \eqref{eq.switching} is unique. Thus
$Y^{(\widehat{a})}_s=R^{(\widehat{a})}_s$ for $s\in[t,T]$. In
particular,
\begin{equation}\label{eq.4.10}
Y^j_t=Y^{(\widehat{a})}_t=R^{(\widehat{a})}_t.
\end{equation}
On the other hand, for every $a\in\widetilde\A^j_t$ we have
\begin{equation*}%\label{eq.4.041}
Y^{\alpha_{n-1}}_{\theta_n}\geq
H^{\alpha_{n-1}}_{\theta_n}(Y_{\theta_n})\geq
h_{\alpha_{n-1},\alpha_n}(\theta_n,
Y^{\alpha_n}_{\theta_n})=Y^{\alpha_n}_{\theta_n}-c^{\alpha_{n-1},
\alpha_n}(\theta_n)
\end{equation*}
and $dK^{(a)}\geq 0,$ so by \eqref{eq.switching}, \eqref{eq.4.4} and the comparison theorem \cite[Proposition 2.1]{klimsiak2014},
\begin{equation*}%\label{eq.4.5}
Y^{(a)}_s\geq R^{(a)}_s,\quad s\in[t,T],\; a\in\widetilde\A^j_t.
\end{equation*}
In particular $Y^{j}_t\geq R^{(a)}_t$ for $a\in\widetilde\A^j_t$,
which when combined with \eqref{eq.4.10}, completes the proof.
\end{proof}

%Observe that the uniqueness of solution of $\ors(\xi,f+dV,H,U)$ follows from Proposition \ref{prop.4.4}, if the process $R^{(\widehat{a})}$ exists.
\begin{wn}\label{cor.4.1}
Let the assumptions of Theorem \ref{prop.4.4} be satisfied. Then the equation $\ors(\xi,f+dV,H,U)$ has a~unique solution.
\end{wn}
\begin{proof}
The existence of a solution of $\ors(\xi,f+dV,H,U)$ follows from Theorem \ref{tw.3}. Uniqueness is a consequence of Theorem \ref{prop.4.4} and Remark \ref{rem.4.3}.
\end{proof}

\section*{Acknowledgements}

This research was supported by Polish National Science Centre (grant  no. \\ 2016/23/B/ST1/01543).


\begin{thebibliography}{99}

\bibitem{Brenann1986} {{M. Brennan and E. Schwartz}}. {{Evaluating natural resource investments.}} {\em J. Business}, {\bf 58}
(1985), 135--157.

\bibitem{CH}
{J.-F. Chassagneux, A. Richou}: {\em Obliquely Reflected Backward Stochastic Differential Equations}, arXiv:1710.08989v3 (2018).

\bibitem{DellMey2}
{{C. Dellacharie and P.-A. Meyer}}, {\em {Probabilities and potential B. Theory of martingales.}}
  (North--Holland Publishing Co., 1982).

\bibitem{Dixit1994} A.~Dixit and R.~Pindyck.
\newblock{\em Investment under Uncertainty.}
\newblock{(Princeton University Press, 1994).}

\bibitem{Grenadier1996} S. Grenadier.
\newblock{The Strategic Exercise of Options: Development Cascades and Overbuilding in Real Estate
Markets.}
\newblock{\em J. Finance} {\bf 51} (1996), 1653--1679.

\bibitem{Grenadier1999} S. Grenadier.
\newblock{Information Revelation through Option Exercise.}
\newblock{\em Rev. Financial Stud.} {\bf 12} (1999), 95--129.

\bibitem{Grenadier2002} S. Grenadier.
\newblock{Option Exercise Games: An Application to the Equilibrium Investment Strategies of Firms.}
\newblock{\em Rev. Financial Stud.} {\bf 15} (2002), 691--721.

\bibitem{HamJean2007}
{S}. {H}amad\`ene and {M}.~{J}eanblanc.
\newblock {On the starting and stopping problem: Application in reversible investments.}
\newblock {\em Math. Oper. Res.}, {\bf 32} (2007), 182--192.

\bibitem{HamZhang2010}
{S}. {H}amad\`ene and {J}.~{Z}hang.
\newblock {Switching problem and related system of reflected backward {SDE}s.}
\newblock {\em Stochastic Process. Appl.}, {\bf 120} (2010), 403--426.

\bibitem{HuTang2010}
Y.~Hu and S.~Tang.
\newblock {{M}ulti--dimensional {BSDE} with oblique reflection and optimal
  switching.}
\newblock {\em Probab. Theory Related Fields}, {\bf 147} (2010), 89--121.

\bibitem{klimsiak2014}
T.~Klimsiak.
\newblock {Reflected {BSDE}s on filtered probability spaces.}
\newblock {\em Stochastic Process. Appl.}, {\bf 125} (2015), 4204--4241.

\bibitem{Klimsiak2016}
{T}. {Klimsiak}.
\newblock {Systems of quasi-variational inequalities related to the switching problem.}
\newblock {\em Stochastic Process. Appl.}, doi:10.1016/j.spa.2018.04.008.

\bibitem{kliroz}
{T}.~{K}limsiak and {A}.~{R}ozkosz.
\newblock {Dirichlet forms and semilinear elliptic equations with measure data.}
\newblock {\em J. Funct. Anal.}, {\bf 265} (2013), 890--925.

\bibitem{kliroz2015}
{{T. Klimsiak and A. Rozkosz}}, {{Obstacle problem
  for semilinear parabolic equations with measure data}}, {\em J. Evol. Equ.} {\bf 15}  (2015), 457--491.

\bibitem{KliRzySlo}
T.~Klimsiak, M.~Rzymowski and L.~S\l omi\'nski.
\newblock {Reflected BSDEs with regulated trajectories.}
\newblock {\em Stochastic Process. Appl.}, doi:10.1016/j.spa.2018.04.011.

\bibitem{McDonald1986} {R.}~McDonald and D.~Siegel.
\newblock{The value of waiting to invest.}
\newblock{\em Quart. J. Econom.}, {\bf 101} (1986), 708--728.

\bibitem{Meyer1962}
{P.-A.} {M}eyer.
\newblock{A decomposition theorem for supermartingales.}
\newblock{\em Illinois J. Math.}, {\bf 6} (1962), 193--205.

\bibitem{Protter}
{P.} {P}rotter.
\newblock{\em Stochastic Integration and Differential Equations.}
(Springer-Verlag, 2004).

\bibitem{TangZhongKoo2013}
{S}. {Tang}, {W}. {Zhong} and {H. K.} {Koo}.
\newblock {Optimal switching of one-dimensional reflected BSDEs and associated
  multidimensional BSDEs with oblique reflection.}
\newblock {\em SIAM J. Control Optim.}, {\bf 49} (2011), 2279--2317.

\bibitem{MT2016}
{M}. {T}opolewski.
\newblock {Reflected BSDEs with general filtration and two completely separated barriers.}
\newblock {\em Probab. Math. Statist.} (to appear), available at arXiv:1611.06745 (2016).

\end{thebibliography}
\end{document}